\documentclass[reqno,12pt]{amsart}

\usepackage{geometry,amsmath,amsfonts,amssymb,amsthm,amscd,mathrsfs,graphicx,enumerate,multicol,wrapfig,subfigure,hyperref,stmaryrd, accents}
\usepackage[all]{xy}
\usepackage{setspace}
\usepackage[T1]{fontenc}
\numberwithin{equation}{subsection}

\newcommand{\ra}{\rightarrow}

\newcommand{\lra}{\longrightarrow}

\newcommand{\p}{\prime}
\newcommand{\pt}{\partial}

\newcommand{\al}{\alpha}

\newcommand{\s}{\sigma}
\newcommand{\vp}{\varphi}
\newcommand{\lam}{\lambda}

\newcommand{\q}{\theta}

\newcommand{\dt}{\delta}

\newcommand{\Zbb}{\mathbb{Z}}
\newcommand{\Pbb}{\mathbb{P}}
\newcommand{\Cbb}{\mathbb{C}}

\theoremstyle{plain} 

\newtheorem{THM}{Theorem}[section]
\newtheorem{DEF}[THM]{Definition}
\newtheorem{EX}[THM]{Example}
\newtheorem{CON}[THM]{Construction}

\newtheorem{PROP}[THM]{Proposition}
\newtheorem{LEM}[THM]{Lemma}
\newtheorem{COR}[THM]{Corollary}
\newtheorem{REM}[THM]{Remark}

\newtheorem*{SNS}{Supermanifold Non-Splitting Theorem}

\newcommand{\bt}{\bullet}

\newcommand{\Aut}{\mathrm{Aut}}

\newcommand{\Oc}{\mathcal{O}}
\newcommand{\red}{\mathrm{red}}
\newcommand{\Hom}{\mathrm{Hom}}

\newcommand{\Jc}{\mathcal{J}}

\newcommand{\Gc}{\mathcal{G}}

\newcommand{\Fc}{\mathcal{F}}

\newcommand{\Abb}{\mathbb{A}}

\newcommand{\Xfr}{\mathfrak{X}}
\newcommand{\Ec}{\mathcal{E}}

\newcommand{\Yfr}{\mathfrak{Y}}
\newcommand{\Ic}{\mathcal{I}}
\newcommand{\Qcl}{\mathcal{Q}}

\usepackage{xcolor}
\definecolor{airforceblue}{rgb}{0.36, 0.54, 0.66}
\definecolor{burgundy}{rgb}{0.5, 0.0, 0.13}
\definecolor{majorelleblue}{rgb}{0.38, 0.31, 0.86}
\definecolor{darkblue}{rgb}{0.0, 0.0, 0.55}

\hypersetup{urlcolor=burgundy, linkcolor=burgundy,
citecolor=darkblue,
colorlinks=true}

\def\dashmapsto{\mathrel{\mapstochar\dashrightarrow}}

\title[Proj. Supersp. Var., Supersp. Quad. and Non-Splitting]{Projective Superspace Varieties, Superspace Quadrics and Non-Splitting
\\}
\author{\small Kowshik Bettadapura}
\date{}

\begin{document}
\maketitle

\begin{abstract} 
This article is a continuation of a previous article which concerned the splitting problem for subspaces of superspaces. We begin with a general account of projective superspaces. Subsequently, we specialise to subvarieties of `positive' projective superspaces. Our main result is: positive, projective superspaces are `normal', in a sense we define. Then, among others, our main application is: smooth, non-reduced, superspace quadric hypersurfaces are non-split.
\end{abstract}

\setcounter{tocdepth}{1}
\tableofcontents

\onehalfspacing

\section{Introduction}

\subsection{Motivation I: Mirror Superspaces}
In the paper by Sethi \cite{SETHI} certain constructions of superspaces were proposed as mirrors to rigid, K\"ahler manifolds appearing in Landau-Ginzberg models. This idea was explored further by Aganagic and Vafa in \cite{AGAVAFA} where the superspace mirror of the projective superspace $\Pbb^{3|4}_\Cbb$ was derived as a quadric in $\Pbb^{3|3}_\Cbb\times \Pbb^{3|3}_\Cbb$. The objective of this article is \emph{not} to study this derivation of mirror superspaces, but rather to comment on the superspaces so derived. We refer to \cite{NOJAMIRROR} where the mirror map for superspaces, among other topics, are studied in more detail.
\\\\
Generally speaking, mirror symmetry relates structures on one space $M$ with structures on another space $\widehat M$, its mirror. For instance, the symplectic structure on $M$ might be suitably interchanged with the complex structure on $\widehat M$. In analogy we can ask, what structures might be interchanged in the mirror symmetry between superspaces? In \cite{BETTEMB} it was proposed, under the mirror symmetry detailed by Sethi; and by Aganagic and Vafa, that the K\"ahler parameter ought to be interchanged with the obstruction to splitting the mirror superspace. This article is hence, in part, an effort to further justify this proposition. We will argue that a large class of superspace varieties obtained in the above-mentioned articles are non-split, i.e., have non-vanishing obstruction class to splitting.

\subsection{Motivation II: Examples}
A well known example of a non-split superspace is the superspace quadric $Q\subset \Pbb^{2|2}_\Cbb$. In homogeneous coordinates $(x^1, \ldots, \q_2)$ it is given by the locus: $(x^1)^2 + (x^2)^2 + (x^3)^2 + \q_1\q_2 = 0$. Arguments showing $Q$ is non-split date to the works of Berezin in \cite{BER}, Manin in \cite{YMAN}, Green in \cite{GREEN} and Onishchik and Bunegina in \cite{ONIBUG}. Another argument, motivated by the observations of Donagi and Witten in \cite{DW1}, was given in \cite{BETTEMB}. More generally, it is natural to wonder when subvarieties given by loci $g + h\q^k = 0$ in projective superspace will be non-split.\footnote{by $h\q^k$ it is meant $\sum h^{i_1, \ldots, i_k} \q_{i_1}\cdots \q_{i_k}$, the sum being over (ordered) multi-indices $(i_1, \ldots, i_k)$.} Adequately resolving this problem could lead to a number of interesting examples of non-split superspaces. 
\\\\
In general, it is a difficult problem to determine whether an analytic superspace is split or not. In \cite{BETTEMB} the splitting problem was studied by reference to embeddings. That is, if we want to study the splitting problem for a superspace $\Yfr$, one method would be to embed it in a split superspace $\Xfr$. One can then study the splitting problem of $\Yfr$ relative to that of $\Xfr$. This perspective was applied in \cite{BETTEMB} to study the splitting problem for certain superspace extensions of rational normal curves. In this article we apply this perspective to subvarieties of projective superspaces more generally.  

We summarise the contents and main results of this article below.

\subsection{Outline and Main Results}
This paper can be broadly divided into three parts, excluding the introductory material in Section \ref{rhf784gf74hf98j0jf309}. The first part concerns the geometry of projective superspaces. The second concerns the splitting problem for subspaces therein, referred to as projective superspace varieties. The third part concerns applications. While the material in the first part of the paper might be well known, the author could not find definitive references for some of the statements made and so they are reproduced here. 

\subsubsection*{Part I: Projective Superspaces}
A well known result, dating at least to Manin in \cite{YMAN}, is that the projective superspace $\Pbb^{m|n}_\Cbb$ is split with structure sheaf $\Oc_{\Pbb_\Cbb^{m|n}}\cong \wedge^\bt \big(\oplus^n\Oc_{\Pbb^m_\Cbb}(-1)\big)$. In Section \ref{rhf87g847hg4899j94jg0} we consider a generalisation to a certain class of weighted projective superspaces $\Pbb^{m|n}_\Cbb(1|\vec b)$ described in Construction \ref{rf894hf894jf80jf09jf0}, where $\vec b = (b_1, \ldots, b_n)$ is an $n$-tuple of integers. In Theorem \ref{rhhf8h98fh93893j3jf0} we show that $\Pbb^{m|n}_\Cbb(1|\vec b)$ is a split superspace with structure sheaf $\Oc_{\Pbb^{m|n}_\Cbb(1|\vec b)}\cong \wedge_{\Oc_{\Pbb^m_\Cbb}}^\bt \big(\oplus_j\Oc_{\Pbb^m_\Cbb}(-b_j)\big)$.

\subsubsection*{Part II: Splitting Projective Superspace Varieties}
By Theorem \ref{rhhf8h98fh93893j3jf0}, projective superspace varieties are subspaces of \emph{split} superspaces. This allows us to apply some of the theory developed in \cite{BETTEMB}. In Theorem \ref{rhf74hf98hf98hf803jf0} we obtain a general characterisation applicable to subvarieties of all weighted projective superspaces $\Pbb^{m|n}_\Cbb(\vec a|\vec b)$ (c.f., Remark \ref{rhf894hf84jf09j4444}), being: when the weighting is `positive', i.e., when $\vec b$ is a tuple of positive integers, then any `homogeneously non-reduced', projective superspace variety (see Definition \ref{gf784g749h89jf0j9033f}) will be `homogeneously non-split' (see Definition \ref{rfh8hf984f8jf0j30}). Roughly put, this means one cannot eliminate the odd variables defining the subvariety (i.e., one cannot split the variety) by automorphisms of the homogeneous coordinate ring alone. As a complement to our study of projective superspaces, we present a short study of their automorphisms in Appendix \ref{g6df36d873dh983hd09j39}. We argue in Theorem \ref{rfh97gf97hf83h0fj390} that certain automorphisms of certain projective superspaces over the projective line can be identified with the general linear group. 
\\\\
In Section \ref{rfb48fg78f78hf983f80f} we establish our main theoretical result in this article. We introduce the notion of `normal embeddings' of superspaces in Definitions \ref{rfj9u4hf89hf89j09fj30} and \ref{cnvjbbiufnido}; Lemma \ref{uiefh98hf983fj309jf03} then clarifies the relation between normal embeddings and non-splitting. Our main result is Theorem \ref{rhf87f9h3f98h3fh03} where we show: any `positive', projective superspace variety is $k$-normal for all $k$, i.e., `normal'.

\subsubsection*{Part III: Applications}
If we are given a projective, superspace variety with defining polynomial equations, how can we confirm whether or not it splits? In order to address this question we give, in Section \ref{rgf87gf87f9h30fj309j},  a more detailed account of the principles underlying the notions and results of the previous section. Section \ref{r8f93hf983f8930} is then concerned with applications to superspace quadrics, which we define generally in Definition \ref{rohf98h89fj039jf093}. Our main result here is Theorem \ref{rfh84hf98hf803jf03} where it is shown: any smooth, non-reduced, superspace quadric hypersurface is non-split. With this result we deduce, in Example \ref{rfh84fg78hf893hf09j903}, non-splitness of a class of mirror superspaces obtained by Sethi in \cite{SETHI}. We next turn our attention to quadrics in products of projective superspaces. Following the work of Lebrun and Poon in \cite{LEBRUN} we present a superspace variant of the classical Segre embedding of products of projective spaces in Theorem \ref{rfh784gf7f9838fj309}. The proof is deferred to Appendix \ref{lfkfkjfhgrye}. Non-splitting of non-reduced, quadric hypersurfaces in products of positive, projective superspaces then follows naturally (see Corollary \ref{djofheofjoieoe}). We conclude the article with Example  \ref{rf784gf87hf83j03j} where non-splitness of the mirror superspace to $\Pbb^{3|4}_\Cbb$, derived by Aganagic and Vafa in \cite{AGAVAFA}, is deduced.

\section{Preliminaries}
\label{rhf784gf74hf98j0jf309}

\subsection{Analytic Superspaces}
We follow the conventions of \cite{GRAUREM}. Let $X$ be a Hausdorff, topological space and $\Oc_X$ a sheaf of commutative, local rings, locally isomorphic to holomorphic functions on $\Cbb^{m+1}$. Then the locally ringed space $(X, \Oc_X)$ defines an analytic space. Points $P\in X$ at which $(X, \Oc_X)$ is not smooth are referred to as singular. Otherwise, if $(X, \Oc_X)$ is smooth at all points $P\in X$ it is referred to as a complex manifold. If we now fix a locally free sheaf of $\Oc_X$-modules $\Ec$, we can form the locally ringed space $(X, \wedge^\bt_{\Oc_X}\Ec)$. This is the prototypical example of an \emph{analytic superspace} and is referred to as a \emph{split superspace}. It depends essentially on $(X, \Oc_X)$ and $\Ec$. Our convention in this article is to set $T_{X, -}^* := \Ec$ and to denote by $S(X, T^*_{X, -})$ the split superspace $(X, \wedge^\bt T^*_{X,-})$.
\\\\
The exterior algebra is an example of a supercommutative ring. In analogy then with analytic spaces, an analytic superspace is a locally ringed space $\Xfr = (X, \Oc_\Xfr)$ where $\Oc_\Xfr$ is a sheaf of supercommutative rings on $X$ that is locally isomorphic to $\wedge^\bt T^*_{X, -}$, for some locally free sheaf $T^*_{X,-}$. We say $S(X, T^*_{X,-})$ is the \emph{split model} associated to $\Xfr$ or that $\Xfr$ is \emph{modelled} on $S(X, T^*_{X, -})$. The superspace $\Xfr$ is itself said to be \emph{split} if it is isomorphic to its split model $S(X, T^*_{X,-})$. 

Otherwise, $\Xfr$ is non-split.

\begin{DEF}\label{rfh78hf98hf8j309}
\emph{Let $\Xfr$ be an analytic superspace modelled on $S(X, T^*_{X, -})$. The underlying analytic space $(X, \Oc_X)$ is called the \emph{reduced space} of $\Xfr$; and the locally free sheaf $T^*_{X, -}$ is referred to as the \emph{odd cotangent sheaf}.}
\end{DEF}

\begin{DEF}
\emph{A pair $(X, T^*_{X, -})$ comprising an analytic space and a locally free sheaf is referred to as a \emph{model}. If $\Xfr$ is modelled on the split superspace $S(X, T^*_{X, -})$, then the pair $(X, T^*_{X, -})$ is said to \emph{model} $\Xfr$.}
\end{DEF}

\noindent
If we wish to refer directly to a given superspace $\Xfr$ we will say $\Xfr$ is a superspace \emph{over} $X$ or \emph{over} $(X, T^*_{X, -})$. We remark here that any analytic space $(X, \Oc_X)$ can be thought of trivially as a superspace upon writing $\Oc_X = \wedge^\bt_{\Oc_X}T^*_{X, -}$, where $T_{X, -}^*$ has rank zero. 

\begin{DEF}\label{r78gf874hf89jf0j3}
\emph{If $\Xfr$ is isomorphic to its reduced space, it is called \emph{reduced}.}
\end{DEF}

\noindent
The terminology in Definition \ref{rfh78hf98hf8j309} is justified by the following natural constructions. For $\Xfr = (X, \Oc_\Xfr)$ an analytic superspace modelled on $(X, T^*_{X, -})$ we have a natural surjection of sheaves $\Oc_\Xfr \twoheadrightarrow \Oc_X$. The kernel $\Jc$ is referred to as the \emph{fermionic ideal}. The structure sheaf $\Oc_X$ of $X$ is therefore recovered from $\Oc_\Xfr$ upon quotienting out the fermionic ideal $\Jc$. This gives an embedding of superspaces $X\subset \Xfr$. Similarly, we recover $T_{X, -}^*$ from $\Jc$ as the quotient $\Jc/\Jc^2$, i.e., $T_{X, -}^* = \Jc/\Jc^2$. Accordingly, Berezin in \cite{BER} refers to $T_{X, -}^*$ as the \emph{conormal sheaf}. We prefer `odd cotangent sheaf' since, in local coordinates $(x|\q)$ for $\Xfr$, local sections of $T_{X, -}^*$ can be represented by K\"ahler differentials $d\q$.

\begin{DEF}\label{rhf874g87h398f30j03}
\emph{A \emph{complex supermanifold} is an analytic superspace whose reduced space is a smooth, i.e., a complex manifold.}
\end{DEF}

\begin{DEF}
\emph{Let $\Xfr$ be complex supermanifold modelled on $S(X, T^*_{X, -})$. The \emph{dimension} of $\Xfr$ is defined by the dimension of $X$ and the rank of $T^*_{X, -}$. Accordingly, we write: $\dim\Xfr = (\dim X\vert\mathrm{rank}~T^*_{X, -})$.}
\end{DEF}

\noindent 
A preliminary classification of analytic superspaces $\Xfr$ begins by identifying its split model. This means characterising its reduced space $(X, \Oc_X)$ and its odd cotangent sheaf $T^*_{X,-}$. Subsequently, one looks to confirm whether $\Xfr$ is split or not. This is the starting point for `obstruction theory' for superspaces.

\subsection{Obstruction Sheaves}
The splitting problem for analytic superspaces involves studying the cohomology of what are termed \emph{obstruction sheaves}. To any model $(X, T^*_{X, -})$ we can assign the obstruction sheaf $\Qcl_{T^*_{X, -}}$. This is a $\Zbb$-graded sheaf of $\Oc_X$-modules, non-zero in degrees $0, \ldots, n$ where $n = \mathrm{rank}~T^*_{X, -}$. In even degree $2k$ we have an isomorphism,
\begin{align}
\Qcl^{(2k)}_{T^*_{X, -}} \cong  T_X\otimes \wedge^{2k}T^*_{X, -}.
\label{4fh84hf983hf89f30}
\end{align}
for $T_X$ the tangent sheaf of $X$. The obstruction sheaves $\Qcl_{T^*_{X, -}}$ appear in the paper by Green in \cite{GREEN} and the works of Berezin, collected in \cite{BER}. The obstruction classes to splitting superspaces $\Xfr$, modelled on $S(X, T^*_{X, -})$, are housed in $H^1(X, \Qcl_{T^*_{X, -}})$. However, not every class therein need be an obstruction to splitting some superspace. This point is explored further in \cite{EASTBRU, OBSTHICK, BETTPHD}. It is shown that $H^1(X, \Qcl_{T^*_{X, -}})$ can be interpreted as a space of thickenings which themselves can, in a suitable sense, be obstructed. 

\begin{REM}
\emph{The obstruction sheaves in odd degree also play an important role in studying the splitting problem. In this article however only the even degree components of $\Qcl_{T^*_{X, -}}$ will be relevant.}
\end{REM}

\noindent
For the ultimate applications in this paper we will make use of the following classical result, which we will refer to as the `Supermanifold Non-Splitting Theorem'. One can find it in the works of Berezin in \cite{BER}. Another proof was given in \cite[Appendix A]{BETTHIGHOBS}. This theorem played an essential role in the deduction of the non-splitness of supermoduli spaces by Donagi and Witten in \cite{DW1}. 

It is as follows:

\begin{SNS}
Let $\Xfr$ be a complex supermanifold modelled on $(X, T^*_{X, -})$. Suppose it admits an atlas which defines a non-vanishing obstruction in degree $2$, i.e., a non-zero element in $H^1\big(X, \Qcl^{(2)}_{T^*_{X, -}}\big)$. Then $\Xfr$ is non-split.\qed
\end{SNS}

\section{Projective Superspaces}
\label{rhf87g847hg4899j94jg0}

\subsection{Split Models over Projective Space}
The following is a classical construction of projective superspace which appears in \cite[p. 195]{YMAN} as part of a more general construction of superspace Grassmannians. It is directly analogous to the construction of projective space as a quotient of Euclidean space by the action of the multiplicative group.

\begin{CON}\label{rhf894hf983h8309j03}
Consider $(m+1|n)$-superspace $\Cbb^{m+1|n}$ endowed with global coordinates $(x^\mu|\q_a)$. The group of units $\Cbb^\times$ acts on $\Cbb^{m+1|n}$ by,
\[
(x^\mu|\q_a) \stackrel{\lam}{\longmapsto} (\lam x^\mu|\lam\q_a)
\]
for $\lam\in \Cbb^\times$. The quotient $\big(\Cbb^{m+1|n}-\{(0|0)\}\big)/\Cbb^\times$ is referred to as complex projective superspace and is denoted $\Pbb_\Cbb^{m|n}$.
\end{CON}

\noindent
In \cite{BETTEMB} it was stated without proof that $\Pbb_\Cbb^{m|n}$ is the split model $S(\Pbb^m_\Cbb, \oplus^n \Oc_{\Pbb^m_\Cbb}(-1)\big)$. This itself is not a new result. We consider a generalisation below.

\begin{CON}\label{rf894hf894jf80jf09jf0}
On $\Cbb^{m+1|n}$ with coordinates $(x^\mu|\q_a)$ consider a weighted action of $\Cbb^\times$ as follows,
\begin{align}
(x^\mu|\q_a) \stackrel{\lam}{\longmapsto} (\lam x^\mu|\lam^{b_a}\q_a)
\label{rfh984hf894f90j0}
\end{align}
for integers $b_1, \ldots, b_n$ and $\lam\in \Cbb^\times$. The quotient of $\Cbb^{m+1|n}-\{(0|0)\}$ by the above action will be denoted $\Pbb_\Cbb^{m|n}(1|b_1, \ldots, b_n)$. For notational convenience, let $\vec b = (b_1, \ldots, b_n)$ and set $\Pbb_\Cbb^{m|n}(1|b_1, \ldots, b_n) = \Pbb_\Cbb^{m|n}(1|\vec b)$. As with $\Pbb^{m|n}_\Cbb$ in Construction $\ref{rhf894hf983h8309j03}$, we refer to $ \Pbb_\Cbb^{m|n}(1|\vec b)$ as `complex projective superspace'. 
\end{CON}

\begin{THM}\label{rhhf8h98fh93893j3jf0}
Let $\vec b = (b_1, \ldots, b_n)$ be an $n$-tuple of integers. There exists an isomorphism of supermanifolds:
\[
\Pbb_\Cbb^{m|n}(1|\vec b)
\cong
S\big(\Pbb_\Cbb^m,
\oplus_j \Oc_{\Pbb_\Cbb^m}(-b_j)
\big)
\]
\end{THM}

\begin{proof}
The assertion of the present theorem is precisely the following:
\begin{enumerate}[(i)]
	\item $\Pbb_\Cbb^{m|n}(1|\vec  b)$ is a superspace;
	\item $\big(\Pbb_\Cbb^{m|n}(1|\vec  b)\big)_\red = \Pbb_\Cbb^m$;
	\item $\Jc/\Jc^2 = \oplus_j \Oc_{\Pbb_\Cbb^m}(-b_j)$;
	\item $\Pbb_\Cbb^{m|n}(1|\vec  b)$ is split.
\end{enumerate}
Hence to prove this theorem we need to confirm (i)---(iv) above. 

(i) A superspace is a locally ringed space $\Xfr$ with sheaf of rings $\Oc_\Xfr$, supercommutative and locally isomorphic to an exterior algebra. Therefore, to show $\Pbb_\Cbb^{m|n}(1|\vec  b)$ is a superspace, we need to show its structure sheaf is a sheaf of supercommutative algebras, locally isomorphic to an exterior algebra. Recall that $\Pbb_\Cbb^{m|n}(1|\vec  b)$ is the quotient of $\Cbb^{m+1|n}$ by the action of $\Cbb^\times$. Let $(x^\mu|\q_j)$ be coordinates on $\Cbb^{m+1|n}$, $\mu = 1, \ldots,m+1$ and $j = 1, \ldots, n$. On the open set $U_\mu = (x^\mu\neq0)$ in $\Abb_\Cbb^{m+1|n}$ we write,
\begin{align*}
z^\nu_{\{\mu\}} = \frac{x^\nu}{x^\mu}
&&
\mbox{and}
&&
\xi^{\{\mu\}}_j = \frac{\q_a}{(x^\mu)^{b_a}}.
\end{align*}
By construction $z^\nu_{\{\mu\}}$ and $\xi^{\{\mu\}}_j$ are even resp., odd, $\Cbb^\times$-invariant, regular functions on $\Abb_\Cbb^{m+1|n}$. The algebra $\Cbb\big[z_{\{\mu\}}|\xi^{\{\mu\}}\big] = \Cbb\big[z^1_{\{\mu\}}, \ldots, z^{m+1}_{\{\mu\}}|\xi^{\{\mu\}}_1, \ldots,\xi^{\{\mu\}}_n\big]$ inherits the structure of an exterior algebra from $\Cbb[x|\q]$. If $\Oc(U_\mu)$ denotes the sheaf of holomorphic functions on $U_\mu$, then $\Cbb\big[z_{\{\mu\}}|\xi^{\{\mu\}}\big] = \Oc(U_\mu)^{\Cbb^\times} = \Oc(\widetilde U_\mu)$, where $\widetilde U_\mu = U/\Cbb^\times \subset \Pbb^{m|n}_\Cbb(1|\vec b)$. In identifying $\Oc_{\Pbb^{m|n}_\Cbb(1|\vec b)}(\widetilde U_\mu) = \Oc(\widetilde U_\mu)$ and observing that $\big(\widetilde U_\mu\big)_{\mu = 1, \ldots,m+1}$ covers $\Pbb^{m|n}_\Cbb(1|\vec b)$ we see that $\Oc_{\Pbb^{m|n}_\Cbb(1|\vec b)}$ is locally an exterior algebra. Hence $\Pbb^{m|n}_\Cbb$ is a superspace.

(ii) $\big(\Pbb_\Cbb^{m|n}(1|\vec  b)\big)_\red$ is characterised by its structure sheaf $\Oc_{\Pbb_\Cbb^{m|n}(1|\vec  b)}/\Jc$, where $\Jc\subset \Oc_{\Pbb_\Cbb^{m|n}(1|\vec  b)}$ is the fermionic ideal. With respect to the covering $(\widetilde U_\mu)$ described above, $\Jc(U_\mu)$ is generated by $(\xi_1^{\{\mu\}}, \ldots, \xi_n^{\{\mu\}})$. Hence $\big(\Oc_{\Pbb_\Cbb^{m|n}(1|\vec  b)}/\Jc\big)(\widetilde U_\mu) \cong \Oc_{\Pbb^m_\Cbb}((\widetilde U_\mu)_\red)$. Since $(\widetilde U_\mu)$ covers $\Pbb_\Cbb^{m|n}(1|\vec  b)$ we deduce an isomorphism of sheaves $\Oc_{\Pbb_\Cbb^{m|n}(1|\vec  b)}/\Jc \cong \Oc_{\Pbb^m_\Cbb}$ and so $\big(\Pbb_\Cbb^{m|n}(1|\vec  b)\big)_\red = \Pbb^m_\Cbb$. 

(iii) Let $\Jc\subset \Oc_{\Pbb_\Cbb^{m|n}(1|\vec  b)}$ be the fermionic ideal. Over $\widetilde U_\mu$ we have $\big(\Jc/\Jc^2\big)(\widetilde U_\mu) \cong \Oc_{\Pbb_\Cbb^m}\big((\widetilde U_\mu)_\red\big)[\xi_1^{\{\mu\}}, \ldots, \xi_n^{\{\mu\}}]/\big(\xi_i^{\{\mu\}}\xi_j^{\{\mu\}})$. On the intersection $ U_\mu\cap U_\nu$ where both $x^\mu\neq0$ and $x^\nu\neq 0$ we have,
\[
\xi_j^{\{\mu\}}
=
\frac{\q_j}{(x^\mu)^{b_j}}
=
\left(\frac{x^\nu}{x^\mu}\right)^{b_j}
\frac{\q_k}{(x^\nu)^{b_j}}
=
\left(\frac{x^\mu}{x^\nu}\right)^{-b_j}
\xi_j^{\{\nu\}}
=
\big(z_{\{\nu\}}^\mu\big)^{-b_j} \xi_j^{\{\nu\}}.
\]
Hence each generator $\xi_j^{\{\mu\}}$ transforms as sections of $\Oc_{\Pbb_\Cbb^m}(-b_j)(U_\mu)$. Therefore, $\Jc/\Jc^2 \cong \oplus_j \Oc_{\Pbb^m}(-b_j)$.

(iv) To see that $\Pbb_\Cbb^{m|n}(1|\vec  b)$ is split, observe that in (i)---(iii) we have constructed an atlas for $\Pbb_\Cbb^{m|n}(1|\vec  b)$ with transition data on $\widetilde U_\mu\cap \widetilde U_\nu$,
\begin{align}
z_{\{\mu\}}^\s = \frac{z^\s_{\{\nu\}}}{z_{\{\nu\}}^\mu}
&&
\mbox{and}
&&
\xi_j^{\{\mu\}}
=
\big(z_{\{\nu\}}^\mu\big)^{-b_j} \xi_j^{\{\nu\}}.
\label{rhf984hf89f0jf093j09j3f}
\end{align}
This is an atlas with trivial obstruction cocycle and so is a split atlas for $\Pbb_\Cbb^{m|n}(1|\vec  b)$. Hence $\Pbb_\Cbb^{m|n}(1|\vec  b)$ is split. 

This theorem now follows.
\end{proof}

\begin{REM}
\emph{Since $\big(\Pbb^{m|n}_\Cbb(1|\vec b)\big)_\red = \Pbb^m_\Cbb$ is a complex manifold (i.e., non-singular, analytic space) we see that $\Pbb^{m|n}_\Cbb(1|\vec b)$ is a complex supermanifold according to Definition \ref{rhf874g87h398f30j03}. When $\vec b = (1, \ldots, 1)$ we recover $\Pbb^{m|n}_\Cbb\cong S(\Pbb^m_\Cbb, \oplus^n \Oc_{\Pbb^m_\Cbb}(-1)\big)$.}
\end{REM}

\noindent
Since any holomorphic vector bundle on $\Pbb_\Cbb^1$ splits into a sum of line bundles, we have the following immediate corollary.

\begin{COR}
Any $(1|n)$-dimensional, split supermanifold with reduced space $\Pbb^1_\Cbb$ 
%Any split model over $\Pbb_\Cbb^1$ and of odd dimension $n$ 
is of the form $\Pbb_\Cbb^{1|n}(1|b_1, \ldots, b_n)$ for some $n$-tuple of integers $(b_1, \ldots, b_n)$.\qed
\end{COR}

\begin{REM}\label{rhf894hf84jf09j4444}
\emph{The action in \eqref{rfh984hf894f90j0} can be generalised to $(x^\mu|\q_a) \stackrel{\lam}{\longmapsto} (\lam^{a^\mu} x^\mu|\lam^{b_a}\q_a)$ for fixed, positive integers $a^1, \ldots, a^{m+1}$. Setting $\vec a = (a^1, \ldots, a^{m+1})$, the quotient of $\Cbb^{m+1|n}$ by this action is $(\vec a|\vec b)$-weighted projective superspace $\Pbb^{m|n}_\Cbb(\vec a|\vec b)$. By construction $\Pbb^{m|n}_\Cbb(\vec a|\vec b)_{\red} = \Pbb^{m|n}_\Cbb(\vec a)$. Since $\Pbb^{m|n}_\Cbb(\vec a)$ is generally a singular variety, $\Pbb^{m|n}_\Cbb(\vec a|\vec b)$ is an example of a singular superspace by Definition \ref{rhf874g87h398f30j03}. Subvarieties of weighted projective superspaces appear in \cite{SETHI} as mirror superspaces. While it would be interesting to study weighted projective superspaces more generally, we refrain from doing so in this article. Our focus is on projective superspaces from Construction \ref{rf894hf894jf80jf09jf0} and varieties therein.}
\end{REM}

\section{Subvarieties and Splitting}
\label{rjhf894h89j988h3}

\subsection{Homogeneous Coordinates}
The ring $\Cbb[x^1, \ldots, x^{m+1}] = \Cbb[x]$ is graded, with graded pieces $\Cbb[x](n)$ comprising homogeneous polynomials of degree $n$. These graded pieces $\Cbb[x](n)$ correspond to global sections of Serre's twisting sheaves $\Oc_{\Pbb_\Cbb^m}(n)$. So for instance, $H^0(\Pbb_\Cbb^m,\Oc_{\Pbb^m_\Cbb}(1))$ is the module of homogeneous coordinates for $\Pbb_\Cbb^m$. In the supercommutative case we consider $\Cbb[x^1, \ldots, x^{m+1}, \q_1, \ldots, \q_n]$, abbreviated to $\Cbb[x|\q]$. If $\q_j$ is positively weighted, with weight $b_j>0$, then $\q_j$ will define local sections $(\xi^{\{\mu\}}_j)$ of $\Oc_{\Pbb_\Cbb^m}(-b_j)$ as we see from Theorem \ref{rhhf8h98fh93893j3jf0}. The sheaf $\Oc_{\Pbb_\Cbb^m}(-b_j)$ has no global sections however so, in stark contrast with the commutative case, $\q_j$ cannot be interpreted as a global section. 
\\\\
Despite this shortcoming, it will nevertheless be instructive to view $\Cbb[x|\q]$ as the homogeneous coordinate ring for projective superspace. Projective superspace varieties then corresponding to homogeneous, prime ideals. Accordingly, with $\Cbb[x]$ the homogeneous coordinate ring of $\Pbb^m_\Cbb$, we view the odd variables $\q$ as `formal parameters' over $\Pbb^m_\Cbb$.

\begin{REM}\label{4fg784gf7hf98h3f803jf30}
\emph{If $\q_j$ is negatively weighted, i.e., $b_j< 0$, then we might view it as a global section over $\Pbb^m_\Cbb$. However, as we will see, there are other drawbacks associated with negative weightings which render them difficult to study.}
\end{REM}

\subsection{Projective Superspace Varieties}
Let $F = (f^\al)$ be a finite collection of \emph{even}, homogeneous polynomials in $\Cbb[x|\q]$, i.e., polynomials of even degree, homogeneous with respect to the action of $\Cbb^\times$. We generically write,
\begin{align}
f^\al(x|\q) = f^\al(x|0) + h^{\al|2}(x)\q^2 + h^{\al|4}(x)\q^4 +\ldots
\label{rhf7hf79h9f83fj03}
\end{align}
where, e.g., by the notation $h^{\al|2} (x)\q^2$ it is meant $\sum_{i,j}h^{\al|ij}(x)\q_i\q_j$ for appropriate polynomial functions $h^{\al|ij}(x)$ preserving homogeneity of $f^\al$. 

\begin{DEF}
\emph{A polynomial $f(x|\q)$ in $\Cbb[x|\q]$ is said to be \emph{irreducible} if $f(x|0)\in \Cbb[x]$ is irreducible.}
\end{DEF}

\noindent
We assume $F = (f^\al)$ is a (finite) collection of even, irreducible, homogeneous polynomials. In supposing the coordinates $x^\mu$ and $\q_j$ are weighted, with weights $1$ and $b_j$ respectively, the locus $\{(x|\q)\mid f^\al(x|\q) = 0,\forall \al\}$ defines a subvariety $V(F)\subset \Pbb^{m|n}_\Cbb(1|\vec b)$. Since each $f^\al(x|\q)$ is homogeneous, then so is $f^\al(x|0)$. Moreover, since $f^\al(x|\q)$ is irreducible then so is $f^\al(x|0)$ by definition. Hence the ideal generated by $\big(f^\al(x|0)\big)$ in $\Cbb[x]$ will define a subvarety $V_0\subset \Pbb_\Cbb^m$. In viewing $V(F)$ as a superspace\footnote{from the material so far presented, it is not yet clear that $V(F)$ will be a superspace as it is unclear whether it will be `locally split'. We address this issue in the subsequent section.}, $V_0$ is its reduced space and its odd cotangent sheaf is the restriction of that of $\Pbb^m_\Cbb(1|\vec b)$ to $V_0$. By Theorem \ref{rhhf8h98fh93893j3jf0}, $V(F)$ is modelled on $\big(V_0, \oplus_j \Oc_{\Pbb_\Cbb^m}(-b_j)|_{V_0}\big)$. In accordance with Definition \ref{rhf874g87h398f30j03}, $V(F)$ is non-singular iff $V_0$ is non-singular.

\begin{REM}
\emph{In \cite{BETTEMB} the subvarieties $V(F)$ are referred to as \emph{even}. This is owing to the fact that the summands of $f^\al$ in \eqref{rhf7hf79h9f83fj03} are all of even degree in the odd variables. More generally, even subspaces have the property that their odd cotangent sheaf is the restriction of the odd cotangent of the ambient superspace. In this article we will only be concerned with even subspaces and thereby drop the prefix `even.' In the case where there is a single polynomial equation $F = (f)$, the subvariety $V(F)$ is referred to as a hypersurface. In \cite{DW1}, hypersurfaces are referred to as \emph{superspace divisors}.}
\end{REM}

\subsection{Splittings}
We begin with a digression on splittings of affine superspaces. 

\subsubsection{Affine Superspace Varieties}
Generally, a splitting of a superspace $\Yfr$ modelled on $(Y, T^*_{Y, -})$ is an isomorphism of $\Yfr$ with its split model $S(Y, T^*_{Y, -})$. Now for a (finite) family $Z = \big(\zeta^\al(x|\q)\big)$ of even, irreducible polynomials in $\Cbb[x|\q]$, the zero set $V(Z)\subset \Abb_\Cbb^{m+1|n}$ is an affine superspace variety. Its coordinate ring is $\Cbb[x|\q]/(\zeta)$. The data modelling $V(Z)$ is $\big( V_0, T_{\Abb_\Cbb^{m+1}, -}^*|_{V_0}\big)$, where $V_0 = \big\{ \zeta^\al(x|0) = 0,\forall \al\big\} \subset \Abb_\Cbb^{m+1}$ and $T_{\Abb_\Cbb^{m+1}, -}^* = \oplus^n \Oc_{\Abb^{m+1}_\Cbb}$. 

\begin{PROP}\label{rhf78hf84hf983f80j30}
Any affine superspace variety is isomorphic to its split model.
\end{PROP}

\begin{proof}
Let $V(Z)\subset \Abb^{m+1|n}_\Cbb$ be an affine superspace variety. Then $\big(V(Z)\big)_\red = V_0$ is an affine variety. Cartan's Theorem B in the complex analytic setting, or Serre's criterion for affineness in the algebraic setting, asserts that  the cohomology of any abelian sheaf on an affine variety is acyclic. Now, any obstruction to splitting $V(Z)$ lies in the first cohomology of the obstruction sheaf on $V_0$, which is an abelian sheaf. Hence this cohomology group vanishes and so any obstructions to splitting $V(Z)$ must vanish. Therefore $V(Z)$ must be split.
\end{proof}

\noindent
Now consider the variety $V_0$ defined by the locus of $\zeta = \big(\zeta^\al(x|0)\big)$. Let $I_\zeta\subset \Cbb[x]$ denote the ideal generated by $\big(\zeta^\al(x|0)\big)$. The coordinate algebra for the split model $S\big(V_0, T_{\Abb_\Cbb^{m|n}, -}^*|_{V_0}\big)$ is $\Cbb[x|\q]/I_\zeta$.\footnote{As an exterior algebra: $\Cbb[x|\q]/I_\zeta = \wedge^\bt_{\Cbb[x]/I_\zeta}\big(\widetilde{J/J^2}\big)$, where $J \subset \Cbb[x|\q]$ is the fermionic ideal; and $J/J^2$ is a $\Cbb[x]$-module and $\widetilde{J/J^2} = (J/J^2)/I_\zeta(J/J^2)$.} If $I_Z \subset \Cbb[x|\q]$ denotes the ideal generated by $Z = \big(\zeta^\al(x|\q)\big)$ then Proposition \ref{rhf78hf84hf983f80j30} implies,
\begin{align}
\frac{\Cbb[x|\q]}{I_Z} 
\cong
\frac{\Cbb[x|\q]}{I_\zeta}.
\label{fbiuuivuhoijijipo3}
\end{align}
Any isomorphism between the algebras in \eqref{fbiuuivuhoijijipo3} is referred to as a \emph{splitting}.

\subsubsection{Projective Superspace Varieties}
We consider the implications of Proposition \ref{rhf78hf84hf983f80j30} now for varieties in projective superspace.

\begin{PROP}\label{rhf784fhf93f893j}
Any variety $V$ in a projective superspace $\Pbb^{m|n}_\Cbb(1|\vec b)$ is itself a superspace.
\end{PROP}

\begin{proof}
In the proof of Theorem \ref{rhhf8h98fh93893j3jf0} we described a system of local coordinates which served to show that $\Pbb^{m|n}_\Cbb(1|\vec b)$ is a split supermanifold. Denote by $(U_\mu)_{\mu = 1, \ldots, m+1}$ this coordinate atlas. For any variety $V\subset \Pbb^{m|n}_\Cbb(1|\vec b)$ note that $V\cap U_\mu$ will be an affine superspace variety. Hence it will be split by Proposition \ref{rhf78hf84hf983f80j30}. Hence $V$ will be locally split. The local splitting will be of the form \eqref{fbiuuivuhoijijipo3} and so $V$ and its split model will be locally isomorphic. It is therefore a superspace.
\end{proof}

\noindent
\begin{REM}\label{rhf73hf93hf98h39}
\emph{A splitting of a projective superspace variety $V\subset \Pbb^{m|n}_\Cbb(1|\vec b)$ is now immediate. It is a consistent choice of local splittings, which exist by Proposition \ref{rhf784fhf93f893j}, that integrate to a global splitting, i.e., that give the same splitting on intersections.}
\end{REM}

\subsection{Homogeneous Splittings}
In contrast to other sections, the results here will apply generally to weighted projective superspaces $\Pbb^{m|n}_\Cbb(\vec a|\vec b)$ (c.f., Remark \ref{rhf894hf84jf09j4444}). Subvarieties of weighted projective superspaces are defined analogously to those of projective superspaces, i.e., by homogeneous, prime ideals in the homogenous coordinate ring. By Definition \ref{r78gf874hf89jf0j3}, a superspace $\Xfr$ is \emph{reduced} if it is isomorphic to its reduced space $\Xfr_\red$. If the rank of the odd cotangent sheaf is non-zero then $\Xfr$ cannot be reduced and so the `interesting' superspaces are all non-reduced. To a projective superspace variety, we consider the  notion `homogeneously non-reduced' in what follows.

\begin{DEF}\label{gf784g749h89jf0j9033f}
\emph{Fix the homogeneous coordinate ring $\Cbb[x|\q]$ and let $F = (f^\al)$ be a finite collection of even, irreducible, homogeneous polynomials in $\Cbb[x|\q]$. The projective, superspace variety $V(F)$ is said to be \emph{homogeneously non-reduced} if there exists at least one $\al$ and $k$ such that $\pt f^\al/\pt \q_k\neq0$.}
\end{DEF}

\begin{REM}
\emph{If a projective, superspace variety is homogeneously reduced, then it will be split as a superspace (c.f., \eqref{fbiuuivuhoijijipo3}).}
\end{REM}

\noindent
In the previous section we described splittings of subvarieties of $\Pbb^{m|n}_\Cbb(1|\vec b)$. Presently, we will consider a weaker form of splitting which is more generally applicable to subvarieties of $\Pbb^{m|n}_\Cbb(\vec a|\vec b)$.

\begin{DEF}\label{rfh8hf984f8jf0j30}
\emph{Let $F = (f^\al)$ be a finite collection homogeneous, even, irreducible polynomials in $\Cbb[x|\q]$. The subvariety $V(F)\subset \Pbb_\Cbb^{m|n}(\vec a|\vec b)$ is \emph{homogeneously split} if there exists an automorphism $\vp$ of $\Cbb[x|\q]$ such that:
\begin{enumerate}[(i)]
	\item the induced map $\overline\vp : \Cbb[x|\q]/(\q^2) \ra \Cbb[x|\q]/(\q^2)$ is the identity;
	\item $\vp$ preserves the weight, i.e., 
	\begin{align*}
	\mathrm{wt.}(x^\mu) = \mathrm{wt.}(\vp(x^\mu))&&\mbox{and}&&\mathrm{wt.}(\q_j) = \mathrm{wt.}(\vp(\q_j));
	\end{align*}
	\item $V\big((F\circ \vp)\big) = \big\{ \big(f^\al\circ \vp\big)(x|\q) = 0\mid\forall \al\big\}$ is homogeneously reduced.
	%, i.e., $\pt (f\circ \vp)/ \pt \q_j = 0$ for all $j$.
\end{enumerate}
If $V(F)$ is \emph{not} homogeneously split, it is said to be \emph{homogeneously non-split}.}
\end{DEF}

\noindent
Restricting to subvarieties of $\Pbb^{m|n}_\Cbb(1|\vec b)$, we have the following relation to splitness following the characterisation in the previous section.

\begin{LEM}\label{rfioeoiejieopsss}
If a subvariety of $\Pbb^{m|n}_\Cbb(1|\vec b)$ is homogeneously reduced, then it is split as a superspace.
\end{LEM}

\begin{proof}
A homogeneous splitting will induce local splittings of the subvariety by Definition \ref{rfh8hf984f8jf0j30}(iii). These local splittings are compatible on intersections by construction. Hence, by Remark \ref{rhf73hf93hf98h39}, we will have a global splitting of the subvariety.
\end{proof}

\noindent
Our objective is to show that homogeneous splittings do not exist if the weighted projective superspace is `positive', defined below.

\begin{DEF}
\emph{The weighted projective superspace $\Pbb^{m|n}_\Cbb(\vec a|\vec b)$ is \emph{positive} if the weights $b_j\in \vec b$ are all positive.}
\end{DEF}

\noindent
We arrive now at the main result of this section.

\begin{THM}\label{rhf74hf98hf98hf803jf0}
Let $\Pbb^{m|n}_\Cbb(\vec a|\vec b)$ be a positive, weighted projective superspace and suppose:
\begin{align}
\mbox{$b_j \geq a^\s$ for all $j$ and $\s$.}
\label{fhueuyebfiuno}
 \end{align}
Then any homogeneously non-reduced subvariety of $\Pbb^{m|n}_\Cbb(\vec a|\vec b)$ will be homogeneously non-split. 
\end{THM}

\begin{proof}
The argument is based on comparing degrees. We will consider hypersurfaces. The generalisation to arbitrary varieties is straightforward. Recall that a hypersurface in $\Pbb^{m|n}_\Cbb(\vec a|\vec b)$ is given by the vanishing locus of an irreducible, homogeneous, even polynomial $f\in \Cbb[x|\q]$. Now, by Definition \ref{rfh8hf984f8jf0j30}(i) a homogeneous splitting $\vp$ will be an automorphism $\vp: \Cbb[x|\q]\ra \Cbb[x|\q]$ given by,
\begin{align}
x^\mu &\longmapsto \vp(x^\mu) = x^\mu + \vp^{\mu|2}(x)\q^2 + \vp^{\mu|4}(x) \q^4 +\cdots
\label{rfj09jf093jf09jf09j30}
\\
\q_j &\longmapsto \vp(\q_j) = \q_j + \vp_{j|3}(x)\q^3 + \vp_{j|5}(x)\q^5 + \cdots
\label{rf894hf8989f398hf93}
\end{align}
Definition \ref{rfh8hf984f8jf0j30}(ii) gives constraints on the coefficients of $\vp$. Suppose $f = g + h\q^k$ for some $k > 0$ and polynomials $g, h\in \Cbb[x]$. If $V(f)$ is homogeneously split then, by Definition \ref{rfh8hf984f8jf0j30}(iii), $V(f\circ \vp)$ is reduced. This means $\pt (f\circ \vp)/\pt \q_j = 0$ for all $j$ by Definition \ref{gf784g749h89jf0j9033f}. As such we must write 
\begin{align}
h = \sum_\s h^\s(x)\frac{\pt g}{\pt x^\s}
\label{fuifiuhiuhihiof3}
\end{align}
for some $h^\s(s)$, in which case $\vp^\s = x^\s - h^\s\q^k + \ldots$, where the ellipses denote terms of order $\q^{k+1}$ and higher. In supposing all of this, we will deduce a contradiction. Firstly, since $f$ is homogeneous we have:
\begin{align}
\deg g = \deg h\q^k.
\label{4hf893hf8jf0j3030}
\end{align}
While the product $h\q^k$ is homogeneous, the individual factors need not be. They are sums of homogeneous polynomials however. Writing out $h\q^k$ explicitly, it is:
\begin{align}
h\q^k
=
\sum_{|I| = k} h^I(x)\q_{b_I}
=
\sum_{|I| = k, \s} h^{I|\s}(x)\frac{\pt g}{\pt x^\s} \q_{b_I},
\label{rfgygf73gf79h389fh30}
\end{align}
where $I = (i_1, \ldots, i_k)$ is a multi-index of length $k$ and $\q_{b_I} = \q_{b_{i_1}}\cdots \q_{b_{i_k}}$. The latter equality in \eqref{rfgygf73gf79h389fh30} follows from \eqref{fuifiuhiuhihiof3}. By \eqref{4hf893hf8jf0j3030} the degree of each summand in \eqref{rfgygf73gf79h389fh30} is constant and equal to $\deg g$. Hence we have:
\begin{align}
\deg g &= \deg h^{I|\s} + \deg g - \deg x^\s + \deg \q_{b_I}
\notag
\\
\iff
\deg \q_{b_I} & =\deg x^\s -  \deg h^{I|\s}.
\label{rgf647f674gf87h398h3}
\end{align}
Since $\deg x^\s$ is positive for all $\s$ and $h^{I|\s}$ is a homogeneous polynomial in $\Cbb[x]$, it follows that the right hand side of \eqref{rgf647f674gf87h398h3} is less-than-or-equal-to $\deg x^\s$. Hence that $\deg \q_{b_I} \leq \deg x^\s$. But now, since $\Pbb^{m|n}_\Cbb(\vec a|\vec b)$ is positive, note that for any $i_\ell\in I$, we have the inequality $b_{i_\ell} < \deg \q_{b_I}$. The inequality is strict since $|I|>0$. Hence, $b_{i_\ell} < \deg x^\s = a^\s$ which contradicts \eqref{fhueuyebfiuno}. The theorem now follows. 
\end{proof}

\begin{EX}
Any homogeneously non-reduced subvariety in $\Pbb_\Cbb^{m|n}(1|\vec b)$, for $\vec b$ positive, will be homogeneously non-split.
\end{EX}

\noindent
A consequence of Theorem \ref{rhf74hf98hf98hf803jf0} is: the property of a variety being homogeneously non-reduced is \emph{independent} of its embedding into the appropriately weighted, positive, projective superspace. However, to clarify, it is more difficult to deduce non-splitness of the variety abstractly. That is, a variety could be abstractly split albeit homogeneously non-split. In the following section we will consider an alternate viewpoint on non-splitting for subvarieties of positive, projective superspaces $\Pbb^{m|n}_\Cbb(1|\vec b)$. We will eventually show that for `quadrics' in a positive, projective superspace, the property of being homogeneously non-reduced implies non-splitness.
\\\\
In Appendix \ref{g6df36d873dh983hd09j39} we have included a brief study of the automorphisms of projective superspaces, building on some of the ideas in this section. As it is irrelevant for the main purposes of this article it is included as an appendix, largely for the sake of interest.

\section{$k$-Normal Embeddings}
\label{rfb48fg78f78hf983f80f}

\subsection{Preliminaries}
We recall some relevant results from \cite{BETTEMB} which we intend on applying here. 
To a smooth embedding of split superspaces $S(Y, T^*_{Y, -})\subset S(X, T^*_{X,-})$ we have the following morphism of exact sequences of sheaves,
\begin{align}
\xymatrix{
0\ar[r]& \Qcl_{T^*_{Y, -}, T^*_{X,-}}\ar[d]
\ar[r] & 
\Qcl_{T^*_{X, -}}\ar[d]
\ar[r] & 
\mathcal R_{T^*_{Y, -}, T^*_{X,-}}\ar[d]
\ar[r] & 
0
\\
0
\ar[r] & 
\Qcl_{T^*_{Y, -}}
\ar[r] & 
\Qcl_{T^*_{X, -}}|_Y
\ar[r] & 
N_{T^*_{Y, -}, T^*_{X,-}}
\ar[r]
& 
0
}
\label{fh7f4hf980f3}
\end{align}
where the vertical maps are the restriction of sheaves on $X$ to $Y\subset X$.

\begin{REM}
\emph{Just like the obstruction sheaf, the sheaves in \eqref{fh7f4hf980f3} are all non-negatively $\Zbb$-graded. They are non-trivial in degrees $0\leq k\leq n$ for $n = \mathrm{rank}~T^*_{X, -}$. As we are only concerned with \emph{even} embeddings, the odd-graded components in \eqref{fh7f4hf980f3} are irrelevant. Indeed, for even embeddings, $N^{(2k+1)}_{T^*_{Y, -}, T^*_{X,-}} = (0)$.}
\end{REM}

\noindent
On cohomology we obtain from \eqref{fh7f4hf980f3} a commutative diagram, a piece of which is:
\begin{align}
\xymatrix{
H^0\big(X, \Qcl_{T^*_{X, -}}\big)\ar[r] \ar[d] & H^0\big(X, \mathcal R_{T^*_{Y, -}, T^*_{X,-}}\big)\ar[r] \ar[d]& H^1\big(X,\Qcl_{T^*_{Y, -}, T^*_{X,-}}\big) \ar[d] 
\\
H^0\big(Y, \Qcl_{T^*_{X, -}}|_Y\big)\ar[r] & H^0\big(Y, N_{T^*_{Y, -}, T^*_{X,-}}\big)\ar[r]^\dt & H^1\big(Y,\Qcl_{T^*_{Y, -}}\big) 
}
\label{rgf784gf79h398fh3}
\end{align}
The following result is proved in \cite{BETTEMB}.

\begin{THM}\label{rf874gf79hf83j093j9f}
Let $\Yfr$ be a supermanifold with split model $S(Y, T^*_{Y, -})$ and suppose there exists a smooth embedding $\iota: \Yfr\subset S(X, T^*_{X, -})$. Then there exists a global section $\phi(\iota)\in H^0\big(Y, N_{T^*_{Y, -}, T^*_{X,-}}\big)$ associated to $\iota$ which, under the boundary map $\dt$ in \eqref{rgf784gf79h398fh3}, maps to an obstruction class to splitting $\Yfr$.\qed
\end{THM}

\subsection{Normal Embeddings}
Observe that the diagram in \eqref{rgf784gf79h398fh3} depends essentially on the embedding of models $(Y, T^*_{Y, -})\subset (X, T^*_{X, -})$.\footnote{From \cite{BETTEMB}, an embedding of models $(Y, T^*_{Y, -})\subset (X, T^*_{X, -})$ is defined by \emph{(i)} an embedding $j: Y\subset X$ and \emph{(ii)} a surjection $j^*T^*_{X, -}\ra T^*_{Y, -}\ra0$. The embedding is smooth if $j$ is smooth.} This leads to the following definitions.

\begin{DEF}\label{rfj9u4hf89hf89j09fj30}
\emph{Let $(Y, T^*_{Y, -})\subset (X, T^*_{X, -})$ be an embedding of models. We say this embedding is \emph{$k$-normal} if the boundary map $\dt$ in \eqref{rgf784gf79h398fh3} is injective in degree $k$, i.e., if $\dt :  H^0\big(Y, N^{(k)}_{T^*_{Y, -}, T^*_{X,-}}\big)\ra H^1\big(Y,\Qcl^{(k)}_{T^*_{Y, -}}\big)$ is injective. If the embedding of models is $k$-normal for all $k>1$, then it is referred to as \emph{normal}.}
\end{DEF}

\noindent
Any embedding of superspaces $\Yfr\subset \Xfr$ begins with a given embedding of models $(Y,T^*_{Y, -})\subset (X, T^*_{X, -})$. This is explained in more detail in \cite{BETTEMB}. We mention it now only in order to justify the following definition.

\begin{DEF}\label{cnvjbbiufnido}
\emph{An embedding of superspaces $\Yfr\subset \Xfr$ is said to be $k$-normal (resp. normal) if the corresponding embedding of models $(Y, T^*_{Y, -})\subset (X, T^*_{X, -})$ is $k$-normal (resp. normal).}
\end{DEF}

\noindent
In the special case where $k = 2$ the Supermanifold Non-Splitting Theorem will imply the following. 

\begin{LEM}\label{uiefh98hf983fj309jf03}
Let $\Yfr\subset S(X, T^*_{X, -})$ be a smooth, $2$-normal embedding and suppose the global section $\phi$ associated to this embedding lies in $H^0\big(Y, N_{T^*_{Y, -};T^*_{X, -}}^{(2)}\big)$. If $\phi\neq0$ then $\Yfr$ is non-split.\qed
\end{LEM}

\subsection{Projective Superspace Varieties}
In the sections to follow we will be more explicit in our description of subvarieties and splittings. Presently, our objective is to prove the following. 

\begin{THM}\label{rhf87f9h3f98h3fh03}
Smooth, positive, projective superspace varieties are normal.
\end{THM}

\begin{proof}
To a smooth embedding $j : Y\subset X$ and a sheaf $\Fc$ on $Y$, we have:
\begin{align}
H^\ell\big(Y, \Fc\big) \cong H^\ell\big(X, j_*\Fc\big).
\label{rhf7hf9h38fh30}
\end{align}
Hence the sheaf cohomology of subspaces $Y$ of $X$ can be calculated on the ambient space $X$. We apply this to the case of a smooth projective variety $V$ of degree $d$. Let $j : V\subset \Pbb^m_\Cbb$ be the smooth embedding. For any abelian sheaf $\Gc$ on $\Pbb^m_\Cbb$ we have the short exact sequence,
\begin{align}
0\lra \Gc(-d)\lra \Gc \lra j_*j^*\Gc \lra0
\label{rjoihhihioiejoe}
\end{align}
where $\Gc(-d)= \Gc\otimes\Oc_{\Pbb_\Cbb^m}(-d)$. Since $j$ is smooth we can identify $j^*\Gc$ with the restriction $\Gc|_V$. By \eqref{rhf7hf9h38fh30} we have the following exact pieces aiding in the calculation of the cohomology of $j^*\Gc = \Gc|_V$:
\begin{align}
%H^\ell\big(\Pbb^m_\Cbb, \Gc(-d)\big)
%\ra
H^\ell\big(\Pbb^m_\Cbb, \Gc\big)
\lra 
H^\ell\big(V, \Gc|_V\big)
\lra 
H^{\ell+1}\big(\Pbb^m_\Cbb, \Gc(-d)\big)
\label{rg784gf87hf98309fj9f}
\end{align}
Now set $\Gc = \Qcl_{T^*_{\Pbb_\Cbb^{m}, -}}$, the obstruction sheaf of the ambient superspace $\Pbb^{m|n}_\Cbb$. By exactness of the rows in \eqref{rgf784gf79h398fh3}, this theorem will follow if we can show
 $H^0(V,  \Qcl_{T^*_{\Pbb_\Cbb^{m}, -}}|_V) = (0)$ in even degree and for any $V$. This is what we will show now. By Theorem \ref{rhhf8h98fh93893j3jf0} we know $T^*_{\Pbb^{m}_\Cbb, -} = \Cbb^n\otimes \Oc_{\Pbb^m_\Cbb}(-1) = \oplus_{j = 1}^n\Oc_{\Pbb^m_\Cbb}(-1)$. From the characterisation of obstruction sheaves in \eqref{4fh84hf983hf89f30} we have:
\begin{align*}
\Qcl_{T^*_{\Pbb_\Cbb^{m}, -}}^{(2k)}
= 
T_{\Pbb^m_\Cbb}\otimes \wedge^{2k}T^*_{\Pbb^{m|n}_\Cbb, -}
=
\oplus^{\binom{n}{2k}}  T_{\Pbb^m_\Cbb}(-2k).
\end{align*}
Now by \eqref{rg784gf87hf98309fj9f} with $\Gc = \Qcl_{T^*_{\Pbb_\Cbb^{m}, -}}$ we have the exact piece,
\begin{align}
\Cbb^{\binom{n}{2k}}\otimes H^0\big(\Pbb^m_\Cbb, T_{\Pbb^m_\Cbb}(-2k)\big)
\lra
\Cbb^{\binom{n}{2k}}& \otimes H^0\big(V, T_{\Pbb^m_\Cbb}(-2k)|_V\big)
\label{titu4j9fj49}
\\
&\lra
\Cbb^{\binom{n}{2k}}\otimes H^1\big(\Pbb^m_\Cbb, T_{\Pbb^m_\Cbb}(-2k-d)\big).
\label{rgf784gf79hf983093j}
\end{align}
Bott's formula asserts the left-most and right-most cohomology groups in \eqref{titu4j9fj49} resp., \eqref{rgf784gf79hf983093j} vanish for $k > 0$ and any $d>0$. Hence $H^0\big(V, T_{\Pbb^m_\Cbb}(-2k)|_V\big) = (0)$ for all $k>0$ and \emph{any} degree $d$, smooth projective variety $V$. We can thus conclude that the embedding of models $(V, T_{V, -}^*)\subset (\Pbb^m_\Cbb, T_{\Pbb_\Cbb, -}^*)$ is normal. This argument applies verbatim with $\Pbb^{m|n}_\Cbb$ replaced by $\Pbb^{m|n}_\Cbb(1|\vec b)$ with $\vec b>0$. This theorem now follows. 
\end{proof}

\begin{REM}
\emph{For subvarieties of non-positive projective superspaces, the left-most cohomology group in \eqref{rgf784gf79hf983093j} need not vanish. As such we cannot readily conclude normality. Where the objectives of this article are concerned, this distinction between positive and non-positive projective superspaces is of fundamental importance (c.f., Remark \ref{4fg784gf7hf98h3f803jf30}).}
\end{REM}

\noindent
Before discussing applications of Theorem \ref{rhf87f9h3f98h3fh03} we digress to explain how exactly one assigns sections in $H^0(Y, N_{T^*_{Y, -}; T^*_{X, -}})$ to subspaces $\Yfr\subset S(X, T^*_{X, -})$ in the case where $\Yfr$ is a projective superspace variety, i.e., when $S(X, T^*_{X, -}) = \Pbb^{m|n}_\Cbb(1|\vec b)$.

\section{Normal Obstruction Sections}
\label{rgf87gf87f9h30fj309j}

\noindent
In Theorem \ref{rf874gf79hf83j093j9f} we see that when a supermanifold $\Yfr$ is embedded in a split superspace $S(X, T^*_{X, -})$, the obstruction class to splitting $\Yfr$ will lie in the image of some global section, which we term below.

\begin{DEF}\label{ruifh398fh983fh0f09j3}
\emph{Let $\Yfr$ be a supermanifold and $S(X, T^*_{X, -})$ some split model. To a given smooth embedding $\iota: \Yfr\subset S(X, T^*_{X, -})$, any global section $\phi\in H^0(Y, N_{T^*_{Y, -}; T^*_{X, -}})$ which maps to an obstruction to splitting $\Yfr$ will be referred to as a \emph{normal obstruction section associated to $\iota$}, or simply a \emph{normal obstruction section} with the embedding $\iota$ understood.}
\end{DEF}

\noindent
In order to get a more explicit description of the obstruction normal section we will need to firstly characterise $N_{T^*_{Y, -}; T^*_{X, -}}$ via more recognisable sheaves. This is done in \cite{BETTEMB} and we will only state the characterisation here: let $(Y, T^*_{Y, -})\subset (X, T^*_{X, -})$ be an even embedding of models\footnote{the embedding $(Y, T^*_{Y, -})\subset (X, T^*_{X, -})$ is \emph{even} if $T^*_{X, -}|_Y \cong T^*_{Y, -}$.} and let $I_Y$ be the ideal sheaf of $Y\subset X$. Then,
\begin{align}
N_{T^*_{Y, -}; T^*_{X, -}}^{(2k)}
\cong 
\mathcal Hom_{\Oc_Y}\big(I_Y/I_Y^2, \wedge^{2k}T^*_{Y, -}\big).
%\label{rfh47f79hf983h0f3j}
\notag
\end{align}
Accordingly, in what follows, we will construct $\Oc_Y$-linear homomorphisms $I_Y/I_Y^2\ra \wedge^{2k}T^*_{Y, -}$ from the data of a superspace variety. These homomorphisms will be the obstruction normal section associated to the variety.

\begin{REM}
\emph{It is instructive compare subspaces of split models with infinitesimal deformations of subschemes. The obstruction normal section is analogous to the class labelling infinitesimal deformations of subschemes, as detailed in \cite{HARTDEF}.}
\end{REM}

\subsection{Affine Superspace Varieties}
Let $\Cbb[x|\q]$ be the coordinate algebra of affine superspace $\Abb_\Cbb^{m+1|n}$ and $F = (f^\al)\in \Cbb[x|\q]$ a finite collection of even, irreducble polynomials. Let $\Yfr = V(F)$ be the variety defined by the locus $\{f^\al = 0\mid \forall \al\}$ and suppose $f^\al = g^\al + h^\al\q^k + \cdots$. The reduced variety $V_0$ is the locus $\{g^\al = 0\mid\forall\al\}$ in $\Abb_\Cbb^{m+1}$. Observe that with $f^\al$ we can tautologically define a lift $g^\al\dashmapsto f^\al$ for all $\al$. This defines a homomorphism of $\Cbb[x]$-modules $I_G\ra \Cbb[x|\q]$ for $I_G\subset \Cbb[x]$ the ideal sheaf generated by $G = (g^\al)$. We now recall some basic properties. Let $J = \big(\q_1, \ldots, \q_n\big)\subset \Cbb[x|\q]$ be the fermionic ideal. Recall $\Cbb[x|\q] = \wedge^\bt_{\Cbb[x]}\big(J/J^2\big)$. Denote by $\pi^k$ the projection $\Cbb[x|\q]\ra \wedge^k_{\Cbb[x]}\big(J/J^2\big)$. Composing this with the lift $G\dashmapsto F$ defines a map $\rho(F): (g^\al)\mapsto (h^\al\q^k)$ and hence a $\Cbb[x]$-module homomorphism $I_G \ra \wedge^k_{\Cbb[x]}\big(J/J^2\big)$. Now recall that we have $p: \Cbb[x]\ra \Cbb[x]/I_G = \Oc(V_0)$ with respect to which we can form the induced module $p_*\wedge^k_{\Cbb[x]}\big(J/J^2\big) = 
%\wedge_{\Oc(V_0)}^j\big(M_\q/I_gM_\q\big) =
 \wedge^k T_{V_0, -}^*$. Composing the lift $\rho(F)$ with the projection $\pi^k$ and the map $\wedge^k_{\Cbb[x]}\big(J/J^2\big)\ra \wedge^kT^*_{V_0, -}$ yields the homomorphism $\overline{\rho(F)} : I_G \ra \wedge^k T_{V_0, -}^*$. By construction $\overline{\rho(F)}$ sends $I_G^2 \ra (0)$. Hence $\overline{\rho(F)} \in \Hom_{\Oc(V_0)}(I_G/I_G^2, \wedge^kT^*_{V_0, -})$. The obstruction normal section associated to the affine superspace variety $V(F)$ is this homomorphism $\overline{\rho(F)}$.

\subsection{Projective Superspace Varieties}
We consider here varieties in projective superspaces of the form $\Pbb_\Cbb^{m|n}(1|\vec b)$. Recall that $\Pbb_\Cbb^{m|n}(1|\vec b)$ is covered by locally affine pieces $\Abb^{m|n}_\Cbb$. Denote by $(U_\mu)_{\mu = 1, \ldots, m+1}$ the affine covering of $\Pbb_\Cbb^{m|n}$ described in the proof of Theorem \ref{rhhf8h98fh93893j3jf0}. If $V\subset \Pbb_\Cbb^{m|n}(1|\vec b)$ is a subvariety, then $V\cap U_\mu\subset \Abb^{m|n}_\Cbb$ is an affine superspace variety. Write $V = V(F)$ for $F = (f^\al)$ a finite collection of even, irreducible, homogeneous polynomials in $\Cbb[x|\q]$. As in the previous section, suppose $f^\al = g^\al + h^\al\q^k + \cdots$, and denote by $V_0\subset \Pbb_\Cbb^m$ the variety $\{g^\al = 0\mid \forall \al\}$. Let $I_{V_0}\subset \Oc_{\Pbb^m_\Cbb}$ be the ideal sheaf defining $V_0\subset \Pbb^m_\Cbb$. With respect to the covering $(U_\mu)$, denote $V(f)_\mu := V(f)\cap U_\mu$ the subvariety in $\Abb^{m|n}_\Cbb$. The construction in the previous section assigns to $V(F)_\mu$ a morphism $\overline{\rho(F)}_\mu : \big(I_{V_0}/I_{V_0}^2\big)(U_\mu) \ra \wedge^kT^*_{V_0, -}(U_\mu)$. Clearly,
\begin{align}
\left.\overline{\rho(F)}\right|_{U_\mu} = \overline{\rho(F)}_\mu.
\label{r784gf78hf983fj030}
\end{align}
We therefore obtain from $V(F)$ a morphism of sheaves $\overline{\rho(F)} : I_{V_0}/I_{V_0}^2\ \ra \wedge^kT^*_{V_0, -}$. This is the normal obstruction section associated to $V(F)$.

\begin{REM}
\emph{The obstruction normal section $\overline{\rho(F)}$ depends on choice of homogeneous coordinate ring $\Cbb[x|\q]$ and hence on the embedding $V(F)\subset \Pbb^{m|n}_\Cbb(1|\vec b)$. If the variety $V(F)$ is positive, i.e., a subvariety of a positive, projective superspace $\Pbb_\Cbb^{m|n}(1|\vec b)$, then homogeneous non-splitting in Theorem \ref{rhf74hf98hf98hf803jf0} implies that $\overline{\rho(F)}$ depends on $\Cbb[x|\q]$ only upto re-scalings by non-zero, complex numbers. In particular, the line defined by $\overline{\rho(F)}$ in the projectivisation $\Pbb\big(\Hom_{\Oc_{V_0}}(I_{V_0}/I_{V_0}^2, \wedge^kT^*_{V_0, -})\big)$ is an invariant of $V(F)$. These statements can also be deduced from normality of positive, projective superspace varieties in Theorem \ref{rhf87f9h3f98h3fh03}.}
\end{REM}

\noindent
At the beginning of Section \ref{rjhf894h89j988h3} it was observed that the odd variables in $\Cbb[x|\q]$, if positively weighted, ought to be viewed as `formal homogeneous coordinates' for projective superspace. Hence any positive, projective, superspace variety $V(F)$ ought to also be viewed formally. The sheaf morphism $\overline{\rho(F)}$ need not necessarily be formal however. It is formal if and only if the sheaf $\mathcal Hom_{\Oc_{V_0}}\big(I_{V_0}/I_{V_0}^2, \wedge^kT^*_{V_0, -}\big)$ does \emph{not} have any global sections.

\begin{PROP}\label{rh89f98f0j309jf03}
Let $V(F)$ be a positive, projective, superspace variety. Suppose there does not exist any global section $\vp$ such $\vp|_{U_\mu} = \overline{\rho(F)}_\mu$. Then $V(F)$ is split.
\end{PROP}

\begin{proof}
Let $(U_\mu)_{\mu = 1, \ldots, m+1}\ra \Pbb^{m|n}_\Cbb(1|\vec b)$ be the affine covering in Theorem \ref{rhhf8h98fh93893j3jf0} with coordinates $\big(z_{\{\mu\}}|\xi^{\{\mu\}}\big)$ on $U_\mu$. Suppose $V(f) = \big(f = 0\big)$ with $f = g + h\q^k$. We will argue, under the hypotheses of this proposition, that $V(f) = (g = 0)\cap (h = 0)$, which means $V(f)$ must be split (c.f., \eqref{fbiuuivuhoijijipo3}). Over each $U_\mu$ we have the affine variety $f\big(z_{\{\mu\}}|\xi^{\{\mu\}}\big) = 0$. We view $V(f) = \big(f(x|\q) = 0\big)$ as being glued together by these affine varieties $V(f)\cap U_\mu = \big(f\big(z_{\{\mu\}}|\xi^{\{\mu\}}\big) = 0 \big)$ in $\Pbb^{m|n}_\Cbb(1|\vec b)$. Indeed, we have: 
\begin{align}
\big(V(f)\cap U_\mu\big)|_{U_\mu\cap U_\mu} = \big(V(f)\cap U_\nu\big)|_{U_\mu\cap U_\mu}.
\label{rjfoif3jf90j390fj30}
\end{align}
Now to each $V(f)\cap U_\mu$ we have the homomorphism $\overline{\rho(f)}_\mu$. By \eqref{rjfoif3jf90j390fj30} we have $\overline{\rho(f)}_\mu|_{U_\mu\cap U_\mu} = \overline{\rho(f)}_\nu|_{U_\mu\cap U_\mu}$. This is precisely the condition that there exist a global section $\vp = \overline{\rho(f)}$ in $\Hom_{\Oc_{V_0}}(I_g/I_g^2, \wedge^kT^*_{V_0, -}) = H^0\big(V(f), \mathcal Hom_{\Oc_{V_0}}(I_g/I_g^2, \wedge^kT^*_{V_0, -})\big)$ satisfying \eqref{r784gf78hf983fj030}. But this violates our assumption that no such global section exists, thereby violating \eqref{rjfoif3jf90j390fj30} unless $\overline{\rho(f)}_\mu = 0$ for all $\mu$, in which case $f = 0$ iff $g = 0$ and $h = 0$. Now if $h =0$, then $V(f)$ will be homogeneously reduced. Hence by Lemma \ref{rfioeoiejieopsss}, $V(f)$ is split.
\end{proof}

\begin{COR}\label{rgf78gf873f9h30fj30}
Let $V(F)$ be a positive, projective, superspace variety and suppose $\overline{\rho(F)}$ is formal. Then $V(F)$ is split.\qed
\end{COR}

\begin{REM}
\emph{Proposition \ref{rh89f98f0j309jf03} and Corollary \ref{rgf78gf873f9h30fj30} clarify why the superspace extension of the rational normal curve of degree $d$ described in \cite{BETTEMB} is split for $d> 2$. When $d = 2$ it is an irreducible quadric and, as we shall see, these will generically be non-split.}
\end{REM}

\noindent
We will conclude with a general characterisation which we intend on applying in the section to follow. To present the characterisation, we introduce the notion of `homogeneous order' for varieties.

\begin{DEF}
\emph{Fix a homogeneous coordinate ring $\Cbb[x|\q]$ and let $F = (f^\al)$ be a finite collection of even, irreducible polynomials. We define the \emph{homogeneous order} of $V(F)$, denoted $\mathrm{ord}(V(F))$, to be the integer:
\[
\mathrm{ord}\left(V(F)\right) = \sup_{\{k\mid k\geq 2\}} \left\{\mbox{$V(F)$ is homogeneously reduced modulo $J^k$}\right\}
\]
where $J\subset \Cbb[x|\q]$ is the fermionic ideal.}
\end{DEF}

\begin{LEM}
Let $V(F)$ be a positive, projective, superspace variety. Then the following are equivalent:
\begin{enumerate}[(i)]
	\item the obstruction normal section of $V(F)$ satisfies, $\overline{\rho(F)}\neq0$;
	\item $V(F)$ is homogeneously non-reduced;
	\item $2\leq \mathrm{ord}\left(V(F)\right) < \mathrm{rank}~J/J^2$.
\end{enumerate}\qed
\end{LEM}

\begin{THM}\label{rhf73f78hf98f0j09fj209jf2}
Let $V(F)$ be a positive, projective, superspace variety with reduced space $V_0$. Suppose $H^0\big(V_0, \mathcal Hom_{\Oc_{V_0}}\big(I_{V_0}/I_{V_0}^2, \wedge^kT^*_{V_0, -}\big)\big)\neq(0)$. Then if $\mathrm{ord}\left(V(F)\right) = k$, the $k$-th obstruction to splitting $V(F)$ non-vanishing. \qed
\end{THM}

\begin{proof}
This follows by combining  Proposition \ref{rh89f98f0j309jf03} with normality of positive, projective superspace varieties in Theorem \ref{rhf87f9h3f98h3fh03}.
\end{proof}

\section{Superspace Quadrics}
\label{r8f93hf983f8930}

\subsection{Quadrics and Non-Splitting}
A classical construct in algebraic geometry is the irreducible quadric. It is a projective variety of degree 2. In projective space $\Pbb^m_\Cbb$ any smooth (i.e., non-singular) quadric can be written in homogeneous coordinates in the form $(x^1)^2 + \cdots + (x^{m+1})^2 = 0$. In projective superspace $\Pbb^{m|n}_\Cbb$ we have the natural analogue of a quadric, being the locus of a degree two, homogeneous polynomial:
\begin{align}
Q(x|\q)
= 
\sum Q_{\mu\nu} x^\mu x^\nu
+
\sum_{ij} Q^{ij}\q_i\q_j
\label{rhf983hf9h389fj30fj3}
\end{align}
where $(Q_{\mu\nu})$ and $(Q^{ij})$ are symmetric resp., antisymmetric matrices over $\Cbb$. In general we propose the following definition of a superspace quadric in projective superspaces:

\begin{DEF}\label{rohf98h89fj039jf093}
\emph{Let $\Cbb[x|\q]$ be the homogeneous coordinate ring of a positive, projective superspace $\Pbb^{m|n}_\Cbb(1|\vec b)$ and let $J = \big(\q_1, \ldots, \q_n\big)\subset \Cbb[x|\q]$ be the fermionic ideal. A positive, projective superspace variety $Q\subset \Pbb^{m|n}_\Cbb(1|\vec b)$, with ideal sheaf $I_Q\subset \Cbb[x|\q]$, is a \emph{superspace quadric} if and only if there exists a unique homomorphism $I_Q \ra \Cbb[x|\q]/J^3$ commuting the following diagram,
\[
\xymatrix{
\ar@{.>}[dr]_{\exists !} I_Q \ar[rr] & & \Cbb[x|\q]
\\
& \Cbb[x|\q]/J^3 \ar[ur]& 
}
\]
where the map $\Cbb[x|\q]/J^3\ra \Cbb[x|\q]$ is a splitting of the exact sequence $J^3\ra\Cbb[x|\q]\ra \Cbb[x|\q]/J^3$.}
\end{DEF}

\noindent
Loci of polynomials of the form \eqref{rhf983hf9h389fj30fj3} are superspace quadrics in $\Pbb_\Cbb^{m|n}$ as per Definition \ref{rohf98h89fj039jf093}. Note however that, more generally, the reduced space of superspace quadrics $Q\subset \Pbb^{m|n}_\Cbb(1|\vec b)$ need not itself be a quadric in $\Pbb^m_\Cbb$.\footnote{e.g., this cannot be the case if, for all pairs of odd weighings $(b_i, b_j)$, we have $b_i + b_j\neq 2$} In what follows we will consider superspace quadric hypersurfaces.

\begin{THM}\label{rfh84hf98hf803jf03}
Let $Q$ be a homogeneously non-reduced, smooth superspace quadric hypersurface. Then $Q$ is non-split.
\end{THM}

\begin{proof}
Let $Q\subset \Pbb^{m|n}_\Cbb(1|\vec b)$ be a smooth superspace quadric. Then $\vec b = (b_j)$ is a tuple of positive integers. By homogeneity, $b_i + b_j$ must be constant for all $i, j$. Say, $b_i + b_j = d$. Let $Q_0\subset \Pbb^m_\Cbb$ be the reduced space of $Q$. By smoothness and homogeneity, $Q_0$ will be a smooth projective variety of degree $d$. If $\Ic_{Q_0}$ is the ideal sheaf defining $Q_0\subset \Pbb^m_\Cbb$, we identify $\Ic_{Q_0} = \Oc_{\Pbb_\Cbb^m}(-d)$. Observe now that, by smoothness, the global sections of obstruction normal sheaf can be calculated on $\Pbb^m_\Cbb$. Hence, we have:
\begin{align}
H^0\big(Q_0,\mathcal Hom_{\Oc_{Q_0}}\big( \Ic_{Q_0}/\Ic_{Q_0}^2,&\wedge^2 T^*_{\Pbb_\Cbb^m, -}|_{Q_0}\big)\big)
\notag
\\
&\cong
\notag
\\
\oplus_j H^0\big(\Pbb^m_\Cbb,& \mathcal Hom_{\Oc_{\Pbb^m_\Cbb}}\big(\Oc_{\Pbb^m_\Cbb}(-d), \Oc_{\Pbb^m_\Cbb}(-d)\big)\big)
\notag
\\
&
\cong 
\Cbb^n.
\label{rhf784gf873hf983h}
\end{align}
In particular, the obstruction normal sheaf admits non-vanishing, global sections. Hence, we can apply Theorem \ref{rhf73f78hf98f0j09fj209jf2}. The present theorem now follows from the Supermanifold Non-Splitting Theorem and Theorem \ref{rhf73f78hf98f0j09fj209jf2}.
\end{proof}

\noindent
As an immediate application we can deduce that the following class of hypersurfaces, appearing in \cite{SETHI} as potential mirrors to Landau-Ginzberg orbifolds, are non-split.

\begin{EX}\label{rfh84fg78hf893hf09j903}
Let $\Cbb[x|\q]$ be the homogeneous coordinate ring for the projective superspace $\Pbb_\Cbb^{3N|2N-2}(1|1,2, \ldots, 1, 2)$, where there are $(N-1)$-many pairs $(1, 2)$. Then the hypersurface defined by $\sum_i x_i^3 + \sum_{k\geq0}\q_{2k+1}\q_{2k+2} = 0$ is a non-singular, superspace quadric. As it is homogeneously non-reduced, it is non-split as a superspace by Theorem $\ref{rfh84hf98hf803jf03}$.
\end{EX}

\subsection{Quadrics in Products}
The category of superspaces admit products. That is, given superspaces $\Xfr = (X, \Oc_\Xfr)$ and $\Xfr^\p = (X^\p, \Oc_{\Xfr^\p})$ we can form the product $\Xfr\times\Xfr^\p$ as the locally ringed space $(X\times X^\p, \Oc_\Xfr\boxtimes\Oc_{\Xfr^\p})$, where $\Oc_\Xfr\boxtimes\Oc_{\Xfr^\p} = p_X^*\Oc_\Xfr\otimes p_{\Xfr^\p}^*\Oc_{\Xfr^\p}$ for projections $p_X : X\times X^\p\ra X$ resp. $p_{X^\p}: X\times X^\p\ra X^\p$. Note that $\Xfr\times\Xfr^\p$ will be a superspace since, locally, we have $\Oc_\Xfr\boxtimes\Oc_{\Xfr^\p}\cong_{\mathrm{loc.}} p_X^*\wedge^\bt T^*_{X, -}\otimes p_{X^\p}^*\wedge^\bt T^*_{X^\p, -} \cong \wedge^\bt p_X^*T^*_{X, -}\otimes \wedge^\bt p_{X^\p}^*T^*_{X^\p, -} \cong \wedge^\bt\big(p_X^*T^*_{X, -}\oplus p_{X^\p}^*T^*_{X^\p, -}\big)$, i.e., that $\Oc_\Xfr\boxtimes\Oc_{\Xfr^\p}$ is locally isomorphic to a sheaf of exterior algebras. If $\Xfr$ and $\Xfr^\p$ are $(m|n)$- resp., $(m^\p|n^\p)$-dimensional, then the product $\Xfr\times\Xfr^\p$ will be $(m+m^\p|n+n^\p)$-dimensional. Clearly, $(\Xfr\times\Xfr^\p)_\red=X\times X^\p$ and the odd cotangent sheaf is $T^*_{X\times X^\p, -} = T^*_{X, -}\boxplus T^*_{X^\p, -} = p_X^*T^*_{X, -}\oplus p_{X^\p}^*T^*_{X^\p, -}$.

\begin{EX}\label{rj89h7fg397fh3o}
\[
\Pbb^{m|n}_\Cbb(1|\vec b) \times \Pbb^{m^\p|n^\p}_\Cbb(1|\vec b^\p) = S\big(\Pbb^m_\Cbb\times\Pbb^{m^\p}_\Cbb, \Oc_{\Pbb_\Cbb^m}(-\vec b) \boxplus\Oc_{\Pbb_\Cbb^{m^\p}}(-\vec b^\p)\big).
\] 
\end{EX}
~\\

\noindent
Let $V$ and $W$ be $m$- and $m^\p$-dimensional, complex vector spaces. The projectivisation of the tensor product map $V\times W\ra V\otimes W$ is referred to as the Segre embedding. Phrased alternately, it is an embedding of projective spaces $\Pbb_\Cbb^m\times \Pbb^{m^\p}_\Cbb \subset  \Pbb^{(m+1)(m^\p+1) - 1}_\Cbb$. This generalises to projective superspaces in the following way.

\begin{THM}\label{rfh784gf7f9838fj309}
There exists a smooth embedding of positive, projective superspaces, 
\[
\Pbb^{m|n}_\Cbb(1|\vec b)\times \Pbb^{m^\p|n^\p}_\Cbb(1|\vec b^\p) \hookrightarrow \Pbb^{m^{\p\p}|n^{\p\p}}_\Cbb(1|\vec b^{\p\p}),
\]
for some $m^{\p\p}, n^{\p\p}$ and $\vec b^{\p\p}$ positive. Explicitly, see \eqref{rfh894f98hf89jf03},  \eqref{virhviurvioejopope} and \eqref{jfkrjvjkrnnveeee} respectively.
\end{THM}

\begin{proof}
The proof involves some aspects of supermanifold theory that have not yet been introduced so we defer it to Appendix \ref{lfkfkjfhgrye}.
\end{proof}

\noindent
A corollary of the embedding in Theorem \ref{rfh784gf7f9838fj309} is the analogue of Theorem \ref{rfh84hf98hf803jf03} for quadrics in products of projective superspaces.

\begin{COR}\label{djofheofjoieoe}
Let $Q$ be a homogeneously non-reduced, quadric hypersurface in a product of positive, projective superspaces. Then $Q$ is non-split.
\end{COR}

\begin{proof}
If $Q\subset \Pbb^{m|n}_\Cbb(1|\vec b)\times \Pbb^{m|n}_\Cbb(1|\vec b^\p)$ is homogeneously non-reduced then, via the embedding in Theorem \ref{rfh784gf7f9838fj309}, it will be a homogeneously non-reduced, quadric hypersurface in $\Pbb^{m^{\p\p}|n^{\p\p}}_\Cbb(1|\vec b^{\p\p})$. If $\Pbb^{m|n}_\Cbb(1|\vec b)\times \Pbb^{m|n}_\Cbb(1|\vec b^\p)$ is a product of positive, projective superspaces, then so is $\Pbb^{m^{\p\p}|n^{\p\p}}_\Cbb(1|\vec b^{\p\p})$. The proof now follows from Theorem \ref{rfh84hf98hf803jf03}.
\end{proof}

\noindent
With Corollary \ref{djofheofjoieoe} above we can deduce non-splitness of the mirror superspace quadric obtained by Aganagic and Vafa in \cite{AGAVAFA}.

\begin{EX}\label{rf784gf87hf83j03j}
If $(x|\q)$ and $(y|\eta)$ are homogeneous coordinates for $\Pbb^{m|n}_\Cbb\times \Pbb^{m|n}_\Cbb$, the quadric defined by the locus $\sum_\mu x^\mu y^\mu + \sum_j \q_j\eta_j = 0$ will be non-split.
\end{EX}

\appendix
\numberwithin{equation}{section}

\section{Automorphisms of Superspaces}
\label{g6df36d873dh983hd09j39}

\subsection*{Preliminaries}
We present a brief study here of the automorphisms of projective superspace. We begin with some preliminary theory, starting with the following definition.

\begin{DEF}
\emph{An automorphism of a superspace $\Xfr$ which fixes the modelling data $(X, T^*_{X, -})$ is referred to as a \emph{framed automorphism}. The group of framed automorphisms of $\Xfr$ is denoted $\Aut_{\mathrm{fr}}\Xfr$.}
\end{DEF}

\noindent
When $\Xfr = S(X, T^*_{X, -})$ is the split model, $\Oc_\Xfr\cong \wedge^\bt T^*_{X, -}$. In this case we have a short exact sequence of sheaves of groups:
\begin{align}
\{1\} \lra \Gc_{T^*_{X, -}} \lra \mathcal Aut_{\Zbb_2}\wedge^\bt T^*_{X, -} \lra \mathcal Aut_{\Oc_X}T^*_{X, -}\lra \{1\}
\label{rnfui4hf3h98fh389}
\end{align}
where $\mathcal Aut_{\Zbb_2}\wedge^\bt T^*_{X, -}$ are the automorphisms of $\wedge^\bt T^*_{X, -}$ as a sheaf of supercommutative algebra which preserves the global, $\Zbb_2$-grading. Denote by $\mathrm{Aut}_0S(X, T^*_{X, -})$ the automorphisms of $S(X, T^*_{X,-})$ which act trivially on the reduced space $S(X, T^*_{X, -})_\red = X$. In terms of sheaves then, $\mathrm{Aut}_0S(X, T^*_{X, -})$ are the global sections of $\mathcal Aut_{\Zbb_2}\wedge^\bt T^*_{X, -}$, i.e, $\Aut_0S(X, T^*_{X, -}) = H^0(X, \mathcal Aut_{\Zbb_2}\wedge^\bt T^*_{X, -})$. The framed automorphisms are then:
\begin{align}
\Aut_{\mathrm{fr}} S(X, T^*_{X, -}) = H^0\big(X, \Gc_{T^*_{X,-}}).
\label{rh784g78fh9fh39}
\end{align}
From the long exact sequence on cohomology induced by \eqref{rnfui4hf3h98fh389} we see that as groups: 
\begin{enumerate}[(i)]
	\item $\Aut_{\mathrm{fr}} S(X, T^*_{X, -})$ is a subgroup of $\Aut_0S(X, T^*_{X, -})$ and;
	\item there exists a natural homomorphism 
	\begin{align}
	\eta: \Aut~ S(X, T^*_{X, -}) \ra \mathrm{GL}(T^*_{X, -}),
	\label{fhfhfhfhfhfhfhsss}
	\end{align}
	where $\mathrm{GL}(T^*_{X, -}) = H^0(X, \mathcal Aut_{\Oc_X}T^*_{X, -})$.
\end{enumerate}
In what follows we will consider projective superspaces.

\subsection*{Automorphisms of Projective Superspace}
If $\Cbb[x|\q]$ denotes the homogeneous coordinate ring of $\Pbb^{m|n}_\Cbb(\vec a|\vec b)$, then automorphisms of $\Pbb^{m|n}_\Cbb(\vec a|\vec b)$ are induced by automorphisms $\vp: \Cbb[x|\q]\ra \Cbb[x|\q]$ which preserve the $\Zbb_2$-grading and the weights of the variables $x$ and $\q$, i.e., that
\begin{align}
\mathrm{wt.}(x^\mu) = \mathrm{wt.}(\vp(x^\mu))&&\mbox{and}&&\mathrm{wt.}(\q_j) = \mathrm{wt.}(\vp(\q_j)).
\label{rhf9hf93hf98j30}
\end{align}
Our objective is to prove:

\begin{THM}\label{fg674gf6gf73h93}
Let $\Pbb^{m|n}_\Cbb(1|\vec b)$ be a positive, projective superspace. The automorphisms $\mathrm{Aut}_0\Pbb^{m|n}_\Cbb(1|\vec b)$ is a subgroup of $\mathrm{GL}_n(\Cbb)$.
\end{THM}

\noindent
En route to proving Theorem \ref{fg674gf6gf73h93} is the following.

\begin{LEM}\label{rnu9h9h30j309j}
Let $\Pbb^{m|n}_\Cbb(1|\vec b)$ be a positive, projective superspace. Then any framed automorphism is trivial, i.e., $\mathrm{Aut}_{\mathrm{fr}}\Pbb^{m|n}_\Cbb(1|\vec b) = \{1\}$.
\end{LEM}

\begin{proof}
%The proof of this lemma is similar to the proof of Theorem \ref{rhf74hf98hf98hf803jf0}. 
Consider the covering of $\Pbb^{m|n}_\Cbb(1|\vec b)$ described in the proof of Theorem \ref{rhhf8h98fh93893j3jf0}. Over an open set $U_\mu$, $\Gc_{T^*_{\Pbb_\Cbb^m, -}}(U_\mu)$ is generated by transformations of the form:
\begin{align}
z^\nu_{\{\mu\}} &\longmapsto z^\nu_{\{\mu\}} + \vp_{\{\mu\}}^{\nu|kl}(z_{\{\mu\}})\xi_k^{\{\mu\}}\xi_l^{\{\mu\}} + \ldots
\label{ncoenonoieniooe}\\
\xi_j^{\{\mu\}} &\longmapsto \xi_j^{\{\mu\}} +\vp^{\{\mu\}|lmn}_j(z^{\{\mu\}}) \xi_l^{\{\mu\}}\xi_m^{\{\mu\}}\xi_n^{\{\mu\}} + \ldots
\label{rhf784gf7h98fh3}
\end{align}
where the summation over the free Latin indices is implied. By \eqref{rh784g78fh9fh39} any framed automorphism is a global section of $\Gc_{T^*_{\Pbb_\Cbb^m, -}}$ and hence will induce a local transformation as in \eqref{ncoenonoieniooe} and \eqref{rhf784gf7h98fh3}. Therefore, in homogeneous coordinates, any framed automorphism must be of the form \eqref{rfj09jf093jf09jf09j30} and \eqref{rf894hf8989f398hf93}. Assuming $\Pbb^{m|n}_\Cbb(1|\vec b)$ is positive, any such automorphism must be trivial as, otherwise, it would violate \eqref{rhf9hf93hf98j30}. The lemma now follows.
\end{proof}

\noindent
\emph{Proof of Theorem $\ref{fg674gf6gf73h93}$}.
By Lemma \ref{rnu9h9h30j309j}, the framed automorphisms of $\Pbb^{m|n}_\Cbb(1|\vec b)$, for $\vec b$ a tuple of positive integers, is trivial. Hence the natural map $\eta: \mathrm{Aut}_0\Pbb^{m|n}_\Cbb(1|\vec b) \ra \mathrm{GL}(T^*_{\Pbb_\Cbb^m, -})$ from \eqref{fhfhfhfhfhfhfhsss} is injective, i.e. $\mathrm{Aut}_0\Pbb^{m|n}_\Cbb(1|\vec b)$ is a subgroup of $\mathrm{GL}(T^*_{\Pbb_\Cbb^m, -})$. It remains to show that $\mathrm{GL}(T^*_{\Pbb_\Cbb^m, -})$ is a subgroup of $\mathrm{GL}_n(\Cbb)$. This is immediate upon recalling: \emph{(i)} from Theorem \ref{rhhf8h98fh93893j3jf0} that $T^*_{\Pbb_\Cbb^m, -} = \oplus_j \Oc_{\Pbb^m_\Cbb}(-b_j)$; and \emph{(ii)} that $\mathrm{GL}(T^*_{\Pbb_\Cbb^m, -}) = H^0(\Pbb^m_\Cbb, \mathcal Aut_{\Oc_{\Pbb^m_\Cbb}}T^*_{\Pbb_\Cbb^m, -})$.
\qed

\begin{REM}\label{rfh98hf983f89j30}
\emph{In the case where $\vec b = b\vec 1 = (b, \ldots, b)$ for some integer $b$, we have: $\mathrm{GL}(T^*_{\Pbb^m_\Cbb,-}) = \mathrm{GL}_n(\Cbb)$.}
\end{REM}

\noindent
As an application we can characterise automorphisms of certain, $(1|n)$-dimensional projective superspaces.

\begin{THM}\label{rfh97gf97hf83h0fj390}
Fix a positive integer $d>0$ and let $d\vec 1 = (d, \ldots, d)$ be an $n$-tuple. Then,
\[
\mathrm{Aut}_0\big(\Pbb^{1|n}_\Cbb(1|d\vec 1)\big) \cong \mathrm{GL}_n(\Cbb).
\]
\end{THM}

\begin{proof}
By Theorem \ref{fg674gf6gf73h93} we know that $\mathrm{Aut}_0\big(\Pbb^{1|n}_\Cbb(1|d\vec 1)\big)$ is a subgroup of $\mathrm{GL}_n(\Cbb)$. With $T^*_{\Pbb^1_\Cbb,-} = \oplus^n \Oc_{\Pbb^1_\Cbb}(-d)$ the odd cotangent sheaf, note that $\mathrm{GL}(T^*_{\Pbb_\Cbb^1, -}) = \mathrm{GL}_n(\Cbb)$ (c.f., Remark \ref{rfh98hf983f89j30}). We therefore have an exact sequence, being the long exact sequence on cohomology induced from \eqref{rnfui4hf3h98fh389}:
\begin{align}
\{1\} 
\lra 
\mathrm{Aut}_0\big(\Pbb^{1|n}_\Cbb(1|d\vec 1)\big)
\stackrel{\eta}{\lra}
 \mathrm{GL}_n(\Cbb)
\stackrel{\pt}{\lra}
\mbox{\v H}^1\big( \Pbb^1_\Cbb, \Gc_{T^*_{X, -}}\big) 
\lra \cdots
\label{rgf78gf873hfh398f3}
\end{align}
where $\al$ is an injective morphism of groups and $\pt$ is a map of pointed sets. In the article \cite{BETTHIGHOBS} the notion of a `good model' $(X, T^*_{X, -})$ was introduced in order to study the class of supermanifolds modelled on $(X, T^*_{X, -})$. It was shown there (see \cite[Theorem 4.3]{BETTHIGHOBS}) that: $(X, T^*_{X, -})$ is a good model if and only if the boundary map $\pt : \mathrm{GL}(T^*_{X, -}) \ra \mbox{\v H}^1(X, \Gc_{T^*_{X, -}})$ from the long exact sequence on cohomology in \eqref{rnfui4hf3h98fh389} is trivial. As an application the model $(\Pbb^1_\Cbb, T^*_{\Pbb^1_\Cbb, -})$, with $T^*_{\Pbb^1_\Cbb,-} = \oplus^n \Oc_{\Pbb^1_\Cbb}(-d)$, $d>0$, was shown to be `good' (see \cite[Theorem 5.5]{BETTHIGHOBS}). Hence $\pt$ in \eqref{rgf78gf873hfh398f3} is trivial. Now note the following: if $A \stackrel{\eta}{\ra} B \ra \{e\}$ is an exact sequence of pointed sets, then $A\stackrel{\eta}{\ra} B$ is surjective as a map of sets. Hence $\eta$ in \eqref{rgf78gf873hfh398f3} will be surjective as a map of pointed sets. As it is also an injective homomorphism of groups, it must therefore be isomorphism of groups. The theorem now follows.
\end{proof}

\section{Proof of Theorem \ref{rfh784gf7f9838fj309}}
\label{lfkfkjfhgrye}

\noindent
A superspace variant of the classical Segre embedding is described in \cite{LEBRUN}. This variant is Theorem \ref{rfh784gf7f9838fj309} for $\vec b, \vec b^\p = \vec 1$. The statement for general $\vec b,\vec b^\p$ (positive) follows from a similar argument as in \cite{LEBRUN}. To give it we need to begin with some preliminary observations. Firstly, from Theorem \ref{rhhf8h98fh93893j3jf0} we know that $\Pbb^{m^{\p\p}|n^{\p\p}}_\Cbb(1|\vec b^{\p\p})$ is a split supermanifold. Secondly, from Example \ref{rj89h7fg397fh3o} we see that the product $\Pbb^{m|n}_\Cbb(1|\vec b)\times \Pbb^{m^{\p}|n^{\p}}_\Cbb(1|\vec b^{\p})$ is also a split supermanifold. Hence this theorem concerns embeddings of split supermanifolds. Very generally we have the following useful lemma.

\begin{LEM}\label{767674ggrui4o4o}
Let $(Y, T^*_{Y, -})$ and $(X, T^*_{X, -})$ be models. Suppose there exists an embedding $j: Y\subset X$ of spaces and a surjection $j^*T_{X, -}^* \ra T^*_{Y, -}$ of odd cotangent sheaves, i.e., an embedding of models $(Y, T^*_{Y, -}) \subset (X, T^*_{X, -})$. Then there exists an embedding of split supermanifolds $S(Y, T^*_{Y, -})\subset S(X, T^*_{X, -})$. 
\end{LEM}

\begin{proof}
This lemma was first observed by Lebrun and Poon in \cite{LEBRUN}. It can be proved by appealing  to a characterisation of embeddings by Donagi and Witten in \cite{DW1}, details of which can be found in \cite{BETTEMB}.  
\end{proof}

\noindent
Recall that we wish to show there exists an embedding of supermanifolds $\Pbb^{m|n}_\Cbb(1|\vec b)\times \Pbb^{m^\p|n^\p}_\Cbb(1|\vec b^\p) \subset \Pbb^{m^{\p\p}|n^{\p\p}}_\Cbb(1|\vec b^{\p\p})$. By Lemma \ref{767674ggrui4o4o} it suffices to show there exists an embedding of models. This entails:
\begin{enumerate}[(i)]
	\item an embedding of reduced spaces $j: \Pbb_\Cbb^m\times \Pbb_\Cbb^{m^\p}\subset \Pbb_\Cbb^{m^{\p\p}}$ and;
	\item a surjection $j^*\Oc_{\Pbb^{m^{\p\p}}_\Cbb}(-\vec b^{\p\p})\ra \Oc_{\Pbb^m_\Cbb}(-\vec b)\boxplus\Oc_{\Pbb^{m^\p}_\Cbb}(-\vec b^\p)$ of odd, cotangent sheaves.
\end{enumerate}
As might be clear, the embedding of reduced spaces $j$ is the classical Segre embedding. It remains to deduce (ii), a surjection of odd cotangent sheaves. So let $j: \Pbb_\Cbb^m\times \Pbb_\Cbb^{m^\p}\subset \Pbb_\Cbb^{m^{\p\p}}$ be the Segre embedding. Inspection of the coordinate description of $j$ reveals the following important isomorphism:
\begin{align}
j^*\Oc_{\Pbb^{m^{\p\p}}_\Cbb}(1) \cong \Oc_{\Pbb_\Cbb^m}(1)\boxtimes \Oc_{\Pbb^{m^\p}_\Cbb}(1)
=
p^*\Oc_{\Pbb_\Cbb^m}(1)
\otimes 
p^{\p*}
\Oc_{\Pbb^{m^\p}_\Cbb}(1)
\label{r8f948hf894fj30}
\end{align}
where $p$ resp., $p^\p$ is the projection of $\Pbb_\Cbb^m\times \Pbb_\Cbb^{m^\p}$ onto the first resp., second factor. Then for some integer $k\geq 0$, the isomorphism in \eqref{r8f948hf894fj30} leads to: 
\begin{align}
j^*\Oc_{\Pbb^{m^{\p\p}}_\Cbb}(k)\cong \Oc_{\Pbb_\Cbb^m}(k)\boxtimes \Oc_{\Pbb^{m^\p}_\Cbb}(k). 
\label{rfh894hf89f3j093jf0}
\end{align}
We will now argue the following.

\begin{LEM}\label{rhf894f0jf903jkf3}
Let $k\in \Zbb $ be non-negative. Set $h_{\Pbb_\Cbb^{m^\p}}(k) := \dim_\Cbb H^0(\Pbb^{m^\p}_\Cbb, \Oc_{\Pbb^{m^\p}_\Cbb}(k))$. There exists a surjection of sheaves $\oplus^{h_{\Pbb^{m^\p}_\Cbb}(k)}j^*\Oc_{\Pbb^{m^{\p\p}}_\Cbb}(-k) \ra \Oc_{\Pbb^{m}_\Cbb}(-k)\ra0$.
\end{LEM}

\begin{proof}
In rearranging \eqref{rfh894hf89f3j093jf0} we have the isomorphism: $j^*\Oc_{\Pbb^{m^{\p\p}}_\Cbb}(-k)\boxtimes \Oc_{\Pbb^{m^\p}_\Cbb}(k)\cong p^*\Oc_{\Pbb^m_\Cbb}(-k)$. Now $\Oc_{\Pbb^{m^\p}_\Cbb}(k)$ is generated by its global sections, which means the natural map $H^0(\Pbb^{m^\p}_\Cbb, \Oc_{\Pbb^{m^\p}_\Cbb}(k))\otimes \Oc_{\Pbb^{m^\p}_\Cbb} \ra \Oc_{\Pbb^{m^\p}_\Cbb}(k)$ is surjective. Taking the tensor product  with $j^*\Oc_{\Pbb^{m^{\p\p}}_\Cbb}(-k)$ then gives the desired surjection.
\end{proof}

\noindent
Now recall that the odd cotangent sheaf of $\Pbb^{m|n}_\Cbb(1|\vec b)$ is $\Oc_{\Pbb^m_\Cbb}(-\vec b) = \oplus_i \Oc_{\Pbb_\Cbb^m}(-b_i)$. If  $\Pbb^{m|n}_\Cbb(1|\vec b)$ is positive, then $b_i$ is positive for each $i$. Then from Lemma \ref{rhf894f0jf903jkf3} we have a surjection: $\bigoplus_i  j^*\Oc_{\Pbb^{m^{\p\p}}_\Cbb}(-b_i)^{h_{\Pbb^{m^\p}_\Cbb}(b_i)}\ra \oplus_i \Oc_{\Pbb_\Cbb^m}(-b_i) = \Oc_{\Pbb^m_\Cbb}(-\vec b)$. A similar statement applies to the odd cotangent sheaf $\Oc_{\Pbb^{m^\p}_\Cbb}(-\vec b^\p)$ of $\Pbb^{m^\p|n^\p}_\Cbb(1|\vec b^\p)$. In using that the odd cotangent sheaf of the product  $\Pbb^{m|n}_\Cbb(1|\vec b)\times \Pbb^{m^\p|n^\p}_\Cbb(1|\vec b^\p)$ is $\Oc_{\Pbb^m_\Cbb}(-\vec b)\boxplus \Oc_{\Pbb^{m^\p}_\Cbb}(-\vec b^\p)$ we therefore have the surjection:
\[
\bigoplus_{i=1, i^\p=1}^{i=n, i^\p=n^\p}
\left(j^*\Oc_{\Pbb^{m^{\p\p}}_\Cbb}(-b_i)^{h_{\Pbb^{m^\p}_\Cbb}(b_i)}
\oplus
j^*\Oc_{\Pbb^{m^{\p\p}}_\Cbb}(-b^\p_{i^\p})^{h_{\Pbb^{m}_\Cbb}(b^\p_{i^\p})}
\right)
\ra 
\Oc_{\Pbb^m_\Cbb}(-\vec b)\boxplus \Oc_{\Pbb^{m^\p}_\Cbb}(-\vec b^\p)
\ra
0.
\]
We can now conclude by Lemma \ref{767674ggrui4o4o} that there exists an embedding of split supermanifolds $\Pbb^{m|n}_\Cbb(1|\vec b)\times \Pbb^{m^\p|n^\p}_\Cbb(1|\vec b^\p) \subset \Pbb^{m^{\p\p}|n^{\p\p}}_\Cbb(1|\vec b^{\p\p})$ and appropriate $n^{\p\p}$ and $\vec b^{\p\p}$. Theorem \ref{rfh784gf7f9838fj309} now follows.
\qed
\\\\
We will be more specific about the embedding data $m^{\p\p}, n^{\p\p}$ and $\vec b^{\p\p}$ in Theorem \ref{rfh784gf7f9838fj309} here. Firstly, since the embedding of reduced spaces is the Segre embedding, we  have:
\begin{align}
m^{\p\p} = (m+1)(m^\p+1)-1. 
\label{rfh894f98hf89jf03}
\end{align}
As for $n^{\p\p}$, note firstly that the odd cotangent sheaf of $\Pbb^{m^{\p\p}|n^{\p\p}}_\Cbb(1|\vec b^{\p\p})$ is the sheaf $\bigoplus_{i=1, i^\p=1}^{i=n, i^\p=n^\p}
\left(\Oc_{\Pbb^{m^{\p\p}}_\Cbb}(-b_i)^{h_{\Pbb^{m^\p}_\Cbb}(b_i)}
\oplus
\Oc_{\Pbb^{m^{\p\p}}_\Cbb}(-b^\p_{i^\p})^{h_{\Pbb^{m}_\Cbb}(b^\p_{i^\p})}
\right)$. As such $n^{\p\p}$ is found by counting the number of summands, which is:
\begin{align}
n^{\p\p} 
&=
\sum_{i=1}^n h_{\Pbb^{m^\p}_\Cbb}(b_i) +
\sum_{i^\p=1}^{n^\p} h_{\Pbb^{m}_\Cbb}(b^\p_{i^\p})
\notag 
\\
&=
\sum_{i=1}^n  \binom{m^\p + b_i}{b_i} +
\sum_{i^\p=1}^{n^\p} \binom{m + b^\p_{i^\p}}{b_{i^\p}^\p}.
\label{virhviurvioejopope}
\end{align}
To describe $\vec b^{\p\p}$ we will need to establish some notation. To an $n$-tuple and $n^\p$-tuple of integers $\vec b = (b_1, \ldots, b_n)$ and $\vec b^\p = (b_1^\p, \ldots, b_{n^\p}^\p)$, we denote by $(\vec b, \vec b^\p)$ the $(n+n^\p)$-tuple of integers $(b_1, \ldots, b_n, b_1^\p, \ldots, b_{n^\p}^\p)$. For integers $k$ and $l$, $l>0$, the expression $(k)_l$ is the $l$-tuple $(k, \ldots, k)$. And, for another pair of integers $k^\p$ and $l^\p$, $l^\p>0$, the expression $((k)_l, (k^\p)_{l^\p})$ is the $(l + l^\p)$-tuple $(k, \ldots, k,k^\p, \ldots, k^\p)$. We now have:
\begin{align}
\vec b^{\p\p} = \big((b_1)_{h_{\Pbb^{m^\p}_\Cbb}(b_1)}, \ldots, (b_n)_{h_{\Pbb^{m^\p}_\Cbb}(b_n)}, (b_{1^\p}^\p)_{h_{\Pbb^{m}_\Cbb}(b^\p_{i^\p})}, \ldots, (b_{n^\p}^\p)_{h_{\Pbb^{m}_\Cbb}(b^\p_{i^\p})}\big).
\label{jfkrjvjkrnnveeee}
\end{align}
If $\vec b$ and $\vec b^\p$ are positive, then clearly $\vec b^{\p\p}$ will be positive.

\begin{EX}\label{rgf784gf87hf89j390fj9f30}
When $\vec b = \vec 1$ and $\vec b^\p = \vec 1$ we recover the superspace variant of the Segre embedding in \cite{LEBRUN} being,
\[
\Pbb^{m|n}_\Cbb\times \Pbb^{m^\p|n^\p}_\Cbb \hookrightarrow \Pbb^{(m+1)(m^\p+1)-1|n (m^\p+1)+n^\p(m+1)}_\Cbb.
\]
For $m= m^\p = n = n^\p = 1$, we have $\Pbb^{1|1}_\Cbb\times \Pbb^{1|1}_\Cbb\subset \Pbb^{3|4}_\Cbb$. If $[x_0: x_1: \q]$ and $[y_0: y_1:\eta]$ are coordinates for each factor, the embedding is given by:
\[
\big([x_0: x_1: \q], [y_0: y_1:\eta]\big) \longmapsto  \big[x_0y_0: x_0y_1: x_1y_0: x_1y_1: x_0\eta:x_1\eta:y_0\q:y_1\q\big].
\]
\end{EX}

\bibliographystyle{alpha}
\bibliography{Bibliography}

\hfill
\\
\noindent
\small
\textsc{
Kowshik Bettadapura 
\\
\emph{Yau Mathematical Sciences Center} 
\\
Tsinghua University
\\
Beijing, 100084, China}
\\
\emph{E-mail address:} \href{mailto:kowshik@mail.tsinghua.edu.cn}{kowshik@mail.tsinghua.edu.cn}

\end{document}